\documentclass[
a4paper,
twoside,
]{amsart}%% {scrbook}%%
\usepackage[pdftex,
colorlinks, % Schrift in Farbe, sonst mit Rahmen
bookmarksnumbered, % Inhaltsverzeichnis mit Numerierung
bookmarksopen, % offnet das Inhaltsverzeichnis
linkcolor=blue, %black, % standard red
citecolor=blue, %black, % standard green
]{hyperref}
\usepackage{setspace}
\usepackage{graphicx}
\usepackage{mathtools}
\usepackage[font=small]{caption}
\usepackage{mathrsfs}
\usepackage{amssymb}
\usepackage {amsmath}
\usepackage{amsthm}
\usepackage{xypic}
\providecommand {\norm}[1] {\lVert#1\rVert}
\providecommand {\bignorm}[1] {\Bigl\lVert#1\Bigr\rVert}

\providecommand {\abs}[1] {\lvert#1\rvert}

\providecommand {\inprod}[1]{\langle #1 \rangle}
\providecommand {\set}[1]{\lbrace #1 \rbrace}

\providecommand {\floor}[1]{\lfloor #1 \rfloor}
\providecommand {\besov}[3]{\Lambda^ {#1} _ {#2} ({#3})}
\providecommand {\cc}[2][p]{\mC^{#1}_{#2}}
\providecommand {\bessel}[2]{\mathcal{P}_ {#1} ({#2})}

\providecommand {\app}[3]{\mathcal{E}^{#1}_{#2} ({#3})}
\providecommand {\inject}{\hookrightarrow}
\providecommand {\inv}[1]{{#1}^{-1}}
\newcommand {\schwartz} {\ensuremath{\mathscr{S}}}
\newcommand {\one} {\mathbf{1}}
\newcommand {\bn} {\ensuremath{\mathbb{N}}}

\newcommand {\br} {\ensuremath{\mathbb{R}}}
\newcommand {\bz} {\ensuremath{\mathbb{Z}}}

\newcommand {\bc} {\ensuremath{\mathbb{C}}}

\newcommand {\bnd} {\ensuremath{\mathbb{N}^d}}
\newcommand {\brd} {\ensuremath{\mathbb{R}^d}}
\newcommand {\btd} {\ensuremath{\mathbb{T}^d}}
\newcommand {\bcd} {\ensuremath{\mathbb{C}^d}}

\newcommand {\bzd}[1] [d] {\ensuremath{\bz^{#1}}}
\newcommand {\mA} {\ensuremath{\mathcal{A}}}
\newcommand {\mD} {\ensuremath{\mathcal{D}}}
\newcommand {\mE} {\ensuremath{\mathcal{E}}}
\newcommand {\mF} {\ensuremath{\mathcal{F}}}

{\newcommand {\mG} {\ensuremath{\mathcal{G}}}
\newcommand {\mX} {\ensuremath{\mathcal{X}}}

\newcommand {\mM} {\ensuremath{\mathcal{M}}}

\newcommand {\mB} {\ensuremath{\mathcal{B}}}

\newcommand {\mC} {\ensuremath{\mathcal{C}}} %%  {\ensuremath{\mathscr{C}}}  %%
\newcommand {\mP} {\ensuremath{\mathcal{P}}}
\newcommand {\mS} {\ensuremath{\mathcal{S}}}
\newcommand {\mT} {\ensuremath{\mathcal{T}}}

\newcommand{\ls}{\lesssim}
\newcommand{\gs}{\gtrsim}
\newcommand {\lone} {\ensuremath{\ell^1}}
\newcommand {\ltwo} {\ensuremath{\ell^2}}

\newcommand {\lp} [1] [p] {\ensuremath{\ell^{#1}(\bzd)}}  %l^p spaces
\newcommand {\lpw}[2] [p] {\ensuremath{\ell^{#1}_{#2}(\bzd)}}  %l^p spaces
\newcommand {\lpz} [1] [p] {\ensuremath{\ell^{#1}(\bz)}}  %l^p spaces

\newcommand {\aps}[2] [p] {\ensuremath{\mathcal{E}_{#2}^#1} }%approximation space
\newcommand {\Ck}[1][k] {\ensuremath{\mC^{#1}}}
\newcommand {\jaffard}[1][w] {\ensuremath{\mC^\infty_{#1}}}

\newcommand {\schur}[2][1] {\ensuremath{\mS^{#1}_{#2}}}
\newcommand {\bopzd}  {\ensuremath{\mathcal B ( \ell^2(\bzd))}}

\newcommand {\bop}  {\ensuremath{\mathcal B ( \ell^2)}}

\DeclareMathOperator{\supp}{supp}

\DeclareMathOperator{\dd}{\mathrm{d}}
\DeclareMathOperator{\diag}{Diag}

\DeclareMathOperator{\id}{\mathrm{id}}

\newcommand \BL {bandlimited}

\newcommand \cont {continuous}
\newcommand \BS {Banach space}
\newcommand \BA {Banach algebra}
\newcommand \MA {matrix algebra}

\newcommand \HMA {homogeneous matrix algebra}
\newcommand \hmg {homogeneous}

\newcommand \odd {off-diagonal decay}

\newcommand \as {approximation space}
\newcommand \IC {inverse-closed}

\newcommand \HZ {H\"{o}lder-Zygmund}
\newcommand \LP {Littlewood-Paley}
\newcommand \DPU {dyadic partition of unity}
\newcommand {\cexp} [1] [{x t}] {\ensuremath{e^{2 \pi i #1}}}
\newcommand { \muleb} [1] {\frac{d {#1}}{#1}}
\newcommand { \mulebd} [1] {\frac{d {#1}}{\abs{#1}^d}}
\newcommand { \mulebdii} [1] {\frac{d {#1}}{\abs{#1}_2^d}}
\newtheorem{prop}{Proposition} [section]
\newtheorem{cor}[prop]{Corollary}
\newtheorem{thm}[prop]{Theorem}

\newtheorem{lem}[prop]{Lemma}

\newtheorem*{theo}{Theorem}

\theoremstyle{definition}
\newtheorem{defn}[prop]{Definition} %[section]

\theoremstyle{remark}
\newtheorem*{rem}{Remark}
\newtheorem*{rems}{Remarks}
\newtheorem{ex}[prop]{Example}

\begin{document}
\title 
{% Smooth Subalgebras of Bessel and Besov Type and  Spectral Invariance\\
% or
%\\
Spectral Invariance  of Besov-Bessel  Subalgebras 
% or \\
% Besov-Bessel Subalgebras are \IC
}
%[Inverse Closed Matrix Algebras with Off Diagonal Decay]{Construction of Inverse Closed Matrix Algebras\\ with Off Diagonal Decay\\ by Bessel and Besov Smoothness}
% [Spectral Invariance in Bessel and Besov Algebras] {Spectral Invariance in Bessel and Besov Algebras\\
% of Matrices with  Off-Diagonal Decay } 
\date{\today} 
\author{ Andreas Klotz}
\address{Faculty of Mathematics\\University of Vienna \\
  Nordbergstrasse 15\\A-1090 Vienna, AUSTRIA} 
\email{andreas.klotz@univie.ac.at} \thanks{ A.~K.~ was supported by  National Research Network S106 SISE of the
Austrian Science Foundation (FWF) and the FWF project P22746N13}

\begin{abstract}
Using principles of the theory of smoothness spaces we give systematic constructions of scales of \IC\ subalgebras of a given \BA\ with the action of a $d$-parameter automorphism group.
In particular we obtain the \IC ness of Besov algebras, Bessel potential algebras and approximation algebras of polynomial order  in their defining algebra.
By a proper choice of the group action these general results can be applied to algebras of infinite matrices and yield \IC\ subalgebras of matrices with \odd\ of polynomial order. Besides alternative proofs of known results we obtain new classes of \IC\ subalgebras of matrices with \odd.

This work is a continuation and extension of results presented in \cite{grkl10}. 
% We investigate two systematic constructions of inverse-closed
% subalgebras of a given Banach algebra or operator algebra $\mA $, both of
% which are inspired by classical principles of approximation theory. 
% The first construction requires a closed derivation or a commutative
% automorphism group on $\mA $ and yields a family of smooth 
% inverse-closed subalgebras of $\mA $ that resemble the usual
% H\"older-Zygmund spaces. The second construction starts
% with a graded sequence of subspaces of \mA\ and yields a class of
% inverse-closed subalgebras that resemble the classical approximation
% spaces. We prove a theorem of Jackson-Bernstein type to show that in
% certain cases both constructions are equivalent. 

% These  results about abstract Banach algebras are applied to algebras
% of infinite matrices with off-diagonal decay. In particular, we obtain
% new and unexpected conditions of off-diagonal decay that are preserved under
% matrix inversion. 
 \end{abstract}
\maketitle
%\onehalfspacing

\section{Introduction}

% \subsection*{Inverse closed subalgebras are useful.}
% \label{sec:inverse-clos-subalg}
%In der Arbeit mit BAnachalgebren is es nuetzlich zu wissen, dass die Subalgebra B einer gegebenen \BA A \IC in \mA ist. 
We aim at systematic constructions of \IC\ subalgebras  of a given \BA\ \mA. 
Recall that a subalgebra \mB\ of  \mA\ is called \IC\ in \mA, if
\begin{equation}
  \label{eq:ic}
   b \in \mB \text{ invertible in } \mA \text{ implies } \inv b \in \mB \,.
\end{equation}
Many equivalent notions for this relation are used in the literature, e.g.,one says that  \mB\ is a \emph{spectral subalgebra} of \mA ~\cite{palmer94}, or  \mB\ is \emph{spectrally invariant} in \mA, see ~\cite{Gro09} for a collection of synonyms.
%Dann koennen wir aus dem  vergleichsweise einfacheren Nachweis der Invertierbarkeit eines Elements aus B in A bereits schliessen, dass dieses Inverse schon in B liegt. 

A prototypical result %concerning spectral invariance 
is  Wiener's Lemma, which states precisely that the Wiener algebra of trigonometric series with absolutely convergent coefficients is \IC\ in the algebra of \cont\ functions on the torus. Another example is the algebra $C^m(X)$ of $m$ times \cont ly differentiable functions on a closed interval $X$, which is \IC\ in $C(X)$ by the iterated quotient rule % is. Here spectral invariance means that the reciprocal of a $m$ times \cont ly differentiable function is in $C^m(X)$ as well, the proof is by the iterated quotient rule
(note that the algebra property of $C^m(X)$ follows from the iterated product rule).  

Our original interest was the construction of \BA s of matrices with \odd\ that are \IC\ in \bop, the bounded operators in $\ell^2$, see~\cite{grkl10}. 
A key result is Jaffard's theorem.
\begin{theo}[{\cite{Jaffard90}}]
  If the entries of the matrix $A$ satisfy $\abs{ A(k,l)} \leq C
  \abs{k-l}^{-r}$ for some $C>0$ and $r>0$, and $A$ is invertible in \bop, then
  $\abs{ \inv A(k,l)} \leq C' \abs{k-l}^{-r}$ for some $C'>0$.
\end{theo}
In other words Jaffard's theorem states that the  \BA\  $\jaffard [r]$, consisting of matrices $A$ with finite norm 
$\norm{A}_{\jaffard [r]}=\sup_{k,l} \abs {A(k,l)} (1+\abs{k-l}^r)$,
is \IC\ in \bop. 
Generalizations and variants of this theorem  have been obtained in, e.g.,~\cite{Bas90, Bas97, GL04a, GR08, Jaffard90}. 
%The proofs of spectral invariance are quite diverse, and no general pattern for constructing \IC\ subalgebras of \bop\ seems to emerge from them.
\\
% \subsection{\IC subalgebras and smoothness}
% \label{sec:ic-subalg-smoothn}
Inverse-closed subalgebras are often related to the concept of smoothness,  an elementary example is again the algebra $C^m(X)$. In more generality, the domain of a densely defined, closed and symmetric derivation on the $C^*$ algebra \mA\ was shown to be \IC\ in \mA\ by Bratteli and Robinson~\cite{bratteli75}, an extension of this result to general \BA s was given in~\cite{MR1221047}. Related concepts of smoothness that define \IC\ subalgebras are the \emph{differentials norms}~\cite{blackcuntz}, the $D_p$ algebras~\cite{MR1221047,kissin94}, or the \emph{Leibniz seminorms}~\cite{Rieffel10}.

Although it might be not immediately obvious the examples of matrices with \odd\ fit into this picture, as \odd\ can be decribed by smoothness conditions. In particular, the formal commutator
\[
\delta(A)= [X,A] \,,
\]
$X=\diag((k)_{k \in \bz})$,  is a derivation on \bop, and its domain defines an algebra of matrices with \odd\ that is \IC\ in \bop~\cite[3.4]{grkl10}.

The proof of the spectral invariance of $\mB \subseteq \mA$ makes often detailed use of some specific properties of the involved algebras. The standard proof for Wiener's Lemma is a prime example of an application of the  Gelfand theory. % , there are ``constructive'' proofs as well~\cite{Newman***}.
 Proofs of the spectral invariance of \BA s of matrices involve the theorem of Bochner-Philips~\cite{Bas90,Baskakov97}, interpolation arguments~\cite{GL04a, Sun05,Sun07a} or commutator estimates~\cite{Jaffard90}.
As it turns out, all of these proof methods use some related concepts of smoothness ~\cite{Klo09}. %

In a previous publication \cite{grkl10} we obtained systematic constructions of \IC\ subalgebras of a given \emph{\BA\ } with additional smoothness conditions.  
In particular it was shown in \cite{grkl10} that
\begin{enumerate}
\item the domain of a (not necessaryly densely defined) closed derivation of a \emph{symmetric} \BA\ \mA\ is \IC\ in \mA,
\item \label{contaut} the subalgebra $C(A)$ of \cont\ elements of a \BA\ \mA\ with $d$-parameter automorphism group $\Psi$ and the associated \HZ\ spaces $\besov \infty r \mA$ are \IC\  in \mA,
\item the \as s of polynomial order of a symmetric \BA\ \mA\ are \IC\ in \mA, where the  approximating subspaces are adapted to the algebra multiplication (see Section~\ref{sec:appr-spac-algebr}).
%In particular this is true  for trigonometric approximation  and for approximation with banded matrices.  
\end{enumerate}
 Applied to subalgebras of  \bop\ the theory yields scales of \IC\ subalgebras of matrices with \odd, including the Banach algebras of matrices used in the literature cited above, but also new classes of \IC\ subalgebras of matrices with \odd\ constructed by approximation with banded matrices.
 
In this article we generalize the approach (\ref{contaut}). 
In its essence this approach is based on  the product and quotient rules of real analysis. The identities
\begin{equation}\label{eq:prodquot}
\begin{split}
  \Delta_t(fg) &= T_t f \Delta_tg + \Delta_tf g \,, \\
  \Delta_t(1/f)&= -\frac{\Delta_tf}{T_tf \cdot f}\,,
\end{split}
\end{equation}
where  $T_tf(x)=f(x-t)$ is the translation operator, imply that the smoothness of $f$ (resp. $g$) is preserved by the product $fg$, and by the reciprocal $1/f$. To give an example, the membership of the \cont\ function $f$ in the \HZ\ space $\besov \infty r {L^\infty(\br)}$ for  $0<r<1$ is defined by the condition  $\norm{\Delta_tf}_\infty \leq C \abs t^r$ for some $C>0$, and Equation (\ref{eq:prodquot}) implies that  $\norm{\Delta_t(1/f)}_\infty \leq C' \abs t^r$ for a constant $C'>0$.

In \cite{grkl10} we have (amongst other things) adapted this approach to more general \BA s. If translation is replaced by the  action of a $d$-parameter automorphism group on the \BA\ \mA,  the noncommutative form of the product and quotient rule~(\ref{eq:prodquot}) impliy that the noncommutative \HZ\ space $\besov \infty r \mA$ is an \IC\ subalgebra of \mA~\cite[3.21]{grkl10}.

In this article we extend this approach  to cover Besov and Bessel potential spaces. In more detail, the organization of the paper is as follows.

% %% for \BA s with the uniformly bounded action of a $d$-parameter automorphism group.
% %We do not assume the group action to be strongly \cont\ on the whole of \mA, but we will need the concept of $C_w$-continuity in order to cover some  examples of \BA s of matrices with \odd.

% \subsection*{Organization of the paper}
% \label{sec:organization-paper}
 After introducing notation we describe some classes of matrices that will serve as examples for the theory to be developed, and we define the approximation spaces needed. In Section \ref{sec:smoothnessBA} we introduce smoothness for a \BA\ \mA\ by the action of a $d$-parameter automorphism group, and review basic properties.  Besov spaces are defined in Section~\ref{sec:besov}, and it is proved that they form \IC\ subalgebras of \mA. A useful property not only for simplifying several proofs but also of conceptual interest is a reiteration theorem for Besov spaces (Theorem~\ref{thm-reiteration}). In ~\cite{Amrein96}
 a similar theorem was proved by interpolation methods for Besov spaces of operators. An examination of the details of this proof shows that it uses similar ingredients (and is of similar complexity) as ours, which is valid in the more general setting of a \BA\ with automorphism group. Instead of interpolation theory the proof given here uses estimates for the moduli of smoothness. In Section~\ref{sec:char-besov-spac} we identify Besov spaces as approximation spaces and obtain a \LP-decomposition of them.

For the discussion of Bessel potential spaces $\bessel r \mA$ in Section~\ref{sec:bess-potent-spac} we introduce the concept of $C_w$-continuity~\cite{Amrein96,Arveson74} in order to cover relevant examples of \BA s of matrices with \odd. A description of  Bessel potential spaces  by  hypersingular integrals is used to prove the algebra properties and the \IC ness of $\bessel r \mA$ in \mA. 

Again, we illustrate the abstract concepts by constructing subalgebras of matrices with \odd. % , the automorphism group on \bop\ given by conjugation with modulation operators $M_t$,
% \begin{equation}
%   \label{eq:homma}
%   \chi_t(A)=M_tAM_{-t} \,,
% \end{equation}
% where $M_tx (k)= \cexp [k \cdot t] x(k)$.
The membership of a matrix $A$ in a Besov or Bessel subalgebra of \bop\  is then equivalent to a form of \odd, so this application of the general theory  describes scales of \IC\ subalgebras of matrices with \odd. 

Long proofs have been moved to the appendices.
\subsection*{Acknowledgments}
\label{sec:acknowledgments}
The author wants to thank Karlheinz Gr\"ochenig for many helpful discussions. %Gero Fendler ***

\section{Resources}
\label{sec:resources}

\subsection{Notation}
\label{sec:notation}
%The cardinality of a finite set $A$ is  $\abs A$.
% The natural numbers are $\bn = \set{1, 2, \dotsc}$, and $\bn_0= \bn \cup \set{0}$. The sets
% \bzd, \brd, \bcd\ denote the $d$-tuples of integers, real and complex numbers, respectively. %
%, the symbol $d$ will alwayshave this meaning. 
The \emph{d-dimensional torus} is $\btd =\brd / \bzd $. % and will be often identified with the unit cube $[0,1)^d$.
Let $\bcd_*=\bcd\setminus\set{0}$, and $\brd_*=\brd\setminus\set{0}$. 
The symbol $\floor x$ denotes the greatest integer smaller or equal to the real number $x$.
 
A multi-index $\alpha= (\alpha _1, \dots , \alpha _d)$ is a $d$-tuple of nonnegative integers. We set
$x^\alpha=x_1^{\alpha_1}\cdots x_d^{\alpha_d}$, and $ D^\alpha f(x)= {\partial_1^{\alpha_1}}\cdots
{\partial_d^{\alpha_d}} f (x) $ is the partial derivative.  
% The binomial coefficients are $\binom{ \alpha} {\beta}=
% \binom{\alpha_1}{ \beta_1} \dotsb \binom{\alpha_1}{ \beta_1}$, and the factorial is $\alpha! = \alpha_1! \dotsm
% \alpha_d!$. 
The degree of $x^\alpha$ is $\abs{\alpha} = \sum _{j=1}^d \alpha _j$, and $\beta \leq \alpha $ means that $\beta _j \leq \alpha _j$ for $j=1, \dots ,d$. 
More generally, $\abs {x}_p=\bigl(\sum_{k=1}^d\abs {x(k)}^p \bigr)^{1/p}$ denotes the $p$-norm on $\bcd$.

%  In Chapter~\ref{cha:gener-carl-class} we need the relation $\abs \alpha
% ! < d^{\abs \alpha} \alpha!$, it can be obtained from the multinomial theorem by setting all $d$ summands to one.
% \\
Positive constants will be denoted by $C$, $C'$,$C_1$,$c$, etc., The same symbol might denote different constants in each equation.
If $f$ and $g$ are positive functions, $ f \asymp g$ means that $C_1 f \leq g \leq C_2 f$. 
We sometimes use the notation $f \ls g$ ($f \gs g$) to express that there
is a constant $C>0$ such that $ f \leq C g$ ($ f \geq C g$).

% For $x$ in \bcd\ and $ 1\leq p \leq \infty$ let $\abs{x}_p$ denote be the $p$-norm of $x$, $\abs{x}$ will be used for
% the $1$-norm.  The vectors $e_k$, $1 \leq k \leq d$, are the standard basis of \bcd.  The standard scalar product on
% \bcd\ is $x \cdot y = \sum_{k=1}^d x_k \bar y_k$.

% More generally, if $\Lambda$ is an arbitrary set, the space $\ell^p(\Lambda)$, $1 \leq p \leq \infty$, consists of the
% sequences $(x_\lambda)_{\lambda \in \Lambda}$, for which the norm
% \[
% \norm{x}_{\ell^p(\Lambda)}=
% \begin{cases}
%   \bigl(\sum_{\lambda \in \Lambda}\abs{x}_\lambda^p \bigr)^{1/p},& \quad p<\infty,  \\
%   \sup_{\lambda \in \Lambda}\abs{x}_\lambda,& \quad p=\infty
% \end{cases}
% \]
% is finite. We will always use the symbol $p'$ to denote the \emph{conjugate exponent} to $p$,  $1 \leq p \leq \infty$, that
% is $1/p'=1-1/p$. %If nothing else is said, the symbols $p$ and $q$ will always be used for $\ell^p$ spaces.

The standard basis of $\lp$ is $e_k =(\delta_{jk})_{j \in \bzd}$,  $\inprod{x,y}=\sum_{k \in \bzd}
x(k){y(k)}$ is the standard dual pairing between $\lp$ and its dual $\lp [p']$.  If $p=2$, we define the scalar product as
$\inprod{x,y}=\sum_{k \in \bzd} x(k)\overline{y(k)}$. This should not lead to confusion.

A submultiplicative weight on  on $\bzd$ is a positive function $v:\bzd \to \br$ such that $v(0)=1$
and %if  there is a constant $C$ such that
$v(x+y) \leq v(x)v(y)$ for $x,y \in \bzd$.  The standard  polynomial weights are
$v_r(x) = (1+\abs x)^r$ for $r \geq 0$. The weighted spaces $\lpw [p] w$ are defined by  the norm 
$\norm{x}_{\lpw [p]  w}=\norm{x w}_{\lp}$. If $w=v_r$ we  will simply write $\norm{x}_{\lpw [p]  r}$.

%The \emph{support} of a sequence $x=(x_\lambda)_{\lambda \in \Lambda}$ is the set of nonzero coordinates:
%$\supp(x)=\set{\lambda \in \Lambda \colon x_\lambda \neq 0}$.
%\\

The Schwartz space of rapidly decreasing functions on \brd\ is denoted by $\schwartz(\brd)$. The Fourier transform of $f
\in \schwartz(\brd)$ is $\mF f (\omega) = \hat f (\omega)=\int_{\brd} f(x) e^{-2 \pi i \omega \cdot x} \; dx$. This definition is
extended by duality to $\schwartz'(\brd)$, the space of tempered
distributions. The same symbols are also used for the Fourier transform on $\bzd$ and $\btd$. 

The \cont\ embedding of the normed space $X$  into  the normed space $Y$ is denoted as $X \inject Y$.
The operator norm of a bounded linear mapping $A \colon X \to Y$  is  $\norm{A}_{X \to Y}$. 
In the special case of operators $A \colon \lp [2] \to \lp [2]$ we write $\norm{A}_{\bopzd}=\norm{A}_{\lp [2] \to \lp
  [2]}$ or simply $\norm{A}_{\bop}$.

We will consider Banach spaces with equivalent norms as equal.

\subsection{Inverse closed subalgebras of Banach algebras}

\label{sec:concepts-from-theory}
% Many parts of our investigations on \odd\ of matrices and their inverses are carried out in the broader
% context of \BA s. Besides standard material we use some less known
% concepts. %
%that might not be so well-known and are
% therefore gathered in this section.

All \BA s are assumed to be \emph{unital}. To verify that a \BS\ \mA\ with norm $\norm{\phantom{i}}_\mA$ is a \BA\ it is sufficient to
prove that $\norm{ab}_\mA \leq C \norm{a}_\mA \norm{b}_\mA$ for some constant $C$.  The expression
$\norm{a}'_\mA=\sup_{\norm{b}_\mA =1}\norm{ab}_\mA$ is  an equivalent norm on \mA\ and satisfies $\norm{ab}'_\mA
\leq \norm{a}'_\mA \norm{b}'_\mA $.

\begin{defn}[Inverse-closedness]
  If $\mA \subseteq \mB$ are  Banach algebras with common multiplication and identity, we call \mA\ \emph{inverse-closed}
  in \mB, if
  \begin{equation} \label{eq_3} a \in \mA \text{ and } a^{-1} \in \mB \quad \text{implies}\quad a^{-1} \in \mA.
  \end{equation}
\end{defn}
Inverse-closedness is equivalent to \emph {spectral invariance}. This means that the spectrum %denoted as
$ \sigma_\mA(a) =\set{\lambda \in \bc \colon a-\lambda \text{ not invertible in } \mA} $ of an element $a \in \mA$ satisfies
\begin{equation*}
  \sigma_\mA(a)= \sigma_\mB (a), \quad \quad \text{ for all } \, a \in \mA.
\end{equation*}
The relation of \IC ness is transitive: If \mA\ is \IC\ in \mB\, and \mB\ is \IC\ in \mC, then \mA\ is \IC\ in \mC.
%
% Inverse-closedness can be defined for general algebras. We will use it for locally convex algebras of matrices in
% Chapter~\ref{cha:gener-carl-class}.
% %
\begin{rem} 
  Spectral invariance is a generalization of \emph{Wiener's Lemma}, which 
  states precisely that the \emph{Wiener algebra} $\mF \ell^1(\bzd)$ of absolutely convergent Fourier series is
  \IC\ in $C(\btd)$. See~\cite{Gro09} for a concise overview of the importance of the concept of \IC ness.
\end{rem}

\subsection{Examples of Matrix Algebras}
\label{sec:exma}
To describe the most common forms of off-diagonal decay, let us fix some notation. 
An infinite matrix $A$ over \bzd\ is a function $A:\bzd \times
\bzd \to \bc$.  The $m$-th side diagonal of $A$ is the matrix $\hat A (m)$ with entries
\begin{equation*}
  \label{eq:8}
  \hat A(m)(k,l)=\begin{cases}
    A(k,l),& \quad k-l=m,\\
    0,     & \quad \text{otherwise}.
  \end{cases}
\end{equation*}
A matrix $A$ is \emph{banded} with bandwidth $N$, if
\begin{equation*}
  \label{eq:bandmatrix}
  A=\sum_{\abs m _\infty\leq N} \hat A (m).
\end{equation*}
% Let us define the most common examples of matrix algebras over \bzd.

 Let $ 1 < p \leq \infty$,  $r > d(1-1/p)$, or  $p=1$, and $r \geq 0$. The space $\cc r$ consists of all matrices $A$ with finite norm
\[
\norm{A}_{\cc r} = \biggl( \sum_{ k \in \bzd} \sum_{l \in \bzd}\abs{A(l,l-k)}^p (1+\abs{k})^{r p}\biggr)^{1/p}
\]
with the standard change for $p=\infty$. 

The following special cases have obtained particular interest.  The \emph{Jaffard algebra} $\mC^\infty_r$ consists of the matrices $A$ for which
 $ |A(k,l)| \leq C (1+|k-l|)^{-r} $, so the norm of $\mC^\infty_r$ describes polynomial decay off the
diagonal. % As the operator norm of $\hat A(k)$ is
% $\|\hat{A}(k)\|_{\bop} = \sup_{l \in \bzd} \abs{A(l,l-k)}$, the norm can also be written as
% \begin{equation}
%   \label{eq:ch2}
%   \|A\|_{\mC^\infty_r} = \sup _{k\in \bz ^d} \|\hat{A}(k)\|_{\bop}  (1+\abs{k})^{r} \, .  
% \end{equation}

The \emph{algebra of convolution-dominated matrices} $\mC^1_r, r\geq0$, (sometimes called the Baskakov-Gohberg-Sj\"ostrand algebra) consists of all
matrices $A$, for which there is a $h \in \ell^1_r(\bzd)$  such that $\abs{A x(k)} \leq h * \abs x (k)$, where $\abs
x$ denotes the vector with components $(\abs{x(k)})_{k \in \bzd}$.
% This is the weighted $\ell^1$-norm of the suprema of the side diagonals.

If $1 < p \leq \infty$ and $r > d(1-1/p)$  or $p=1$ and  $r\geq0$ the \emph{Schur algebra} $\mS_r^p$ is defined by the norm
\[%\begin{equation}
  \label{eq:schurnorm}
  \norm{A}_{\mS^p_r} = \max \Big\{ \sup_{k\in \bzd} \bigl(\sum_{l\in \bzd}  \abs{A(k,l)}^pv_r(k-l)^p \bigr)^{1/p},
  \sup _{l\in \bzd}\bigl(\sum_{k\in \bzd} \abs{A(k,l)}^p v_r(k-l)^p \bigr)^{1/p}\Big\}
\]%\end{equation}
with the standard change for $p= \infty$. %If then $\mS^p_r$ is a \BA.
\begin{rems}% \hspace{1cm}
  %\begin{enumerate}
  (1) The scales \schur [p] r and $\mC^p_r$ are identical at the endpoint $p=\infty$, i.e. $\schur [\infty] r = \mC^\infty_r$.
  (2) It follows immediately from the definitions that ${\mC^p_r} \inject {\mS^p_r} $.
 % \end{enumerate}
\end{rems}
     
We note that the norms above depend only on the absolute values of the matrix entries. Precisely,  a matrix norm on $\mA $ is called
\emph{solid}, if $B \in \mA$ and $\abs{A(k,l)} \leq \abs{B(k,l)}$ for all $k,l$ implies $A \in \mA$ and $\norm{A}_{\mA}\leq\norm{B}_{\mA}$.  In
particular, for a solid norm we have $\| \, |A| \|_{\mA } = \|A\|_{\mA}$, where $|A|$ is the matrix with entries $|A|(k,l) = |A(k,l)|$ for $k,l\in
\mathbb{Z}^d $.

The following result summarizes the main properties of the matrix classes $ \mC^p_r$ and $\mS^p_r $. See \cite{Bas90, GL04a,Jaffard90,Sun05} for
proofs.
\begin{prop} \label{fundamental} Assume that $\mA $ is one of the matrix classes $\mC^p _r$ or $\mS^p _r $ for $r>d(1-1/p)$, if $p>1$, and $r \geq 0$, if
  $p=1$.
 Then $\mA $ is a solid Banach $*$-algebra with respect to matrix multiplication and taking adjoints as the involution.
     Every $\mA $ is continuously embedded into the algebra $\mB (\ell ^p(\bzd)) $ of bounded operators on $\ell ^p(\mathbb{Z}^d )$ for $ 1\leq
    p \leq\infty$.
     With the exception of $\schur [1] 0$ \cite{tessera10} every class $\mA $ is inverse-closed in $\mB (\ell ^p(\bzd)) $, $ 1\leq p \leq\infty$. In particular, $\mA $ is symmetric.
  % \begin{enumerate} 
  %   \item Then $\mA $ is a solid Banach $*$-algebra with respect to matrix multiplication and taking adjoints as the involution.
  %   \item Every $\mA $ is continuously embedded into the algebra $\mB (\ell ^p(\bzd)) $ of bounded operators on $\ell ^p(\mathbb{Z}^d )$ for $ 1\leq
  %   p \leq\infty$.
  %   \item Every $\mA $ is inverse-closed in $\mB (\ell ^p(\bzd)) $,$ 1\leq p \leq\infty$. In particular, $\mA $ is symmetric.
  % \end{enumerate}
\end{prop}
In the sequel we will construct algebras that are inverse-closed in one of the standard algebras $ \mC^p _r, \mS^p _r$, and are therefore  inverse-closed in $\bopzd$ by Proposition~\ref{fundamental}.

% \begin{rems}
%   Inverse-closedness of $\jaffard [r]$ in \bop\ has been proved by Jaffard~\cite{Jaffard90} and Baskakov~\cite{Bas90,Bas97} , a very simple proof
%   is due to Sun~\cite{Sun05}. A proof for a more general class of weights is by Baskakov~\cite{Bas97}; Gr\"ochenig and Leinert give a different
%   proof in~\cite{MR2204052}.
%
%   The \IC ness of $\baskakov [0]$ in \bop\ was proved by Gohberg, Kaashoek and Woerdeman in~\cite{GKW89}, and by Sj\"ostrand in~\cite{Sjo95}.
%   Baskakov~\cite{Bas97} identified the most general class of weights, such that $\baskakov [v]$ is \IC\ in \bop. A complete characterization can be
%   found in \cite{GR08}.
% \end{rems}
We generalize the definitions above.
\begin{defn}
  A \emph{\MA} \mA\ (over \bzd) is a \BA\ of matrices that is continuously embedded in \bopzd.
\end{defn}
We drop the reference to the index set \bzd\ whenever
possible.
% We gather some simple properties of \MS s and algebras.
% \end{defn}
\begin{lem} %\hspace{1cm}
  % \begin{enumerate}
  If \mA\ is a \MA, the selection of matrix elements is \cont.% $\abs{A(k,l)} \le C \norm{A}_\mA$.
 \end{lem}
\begin{proof}
  $\abs{A(k,l)}=\abs{\inprod{Ae_k,e_l}} \le \norm{A}_{\bop} \le C \norm{A}_\mA$.
 \end{proof}

 \subsection{Approximation Spaces and Algebras}
 \label{sec:appr-spac-algebr}

Let  the index set $\Lambda$  be either $\br^+_0$ or $\bn_0$. An \emph{approximation scheme} on the \BA\ \mA\
is a family $(X_\sigma)_{\sigma \in \Lambda}$ of closed subspaces of \mA\ that satisfy
$ X_0=\set{0}$, $ X_\sigma \subseteq X_\tau$  for $\sigma \leq \tau$,  and
$  X_\sigma \cdot X_\tau\subseteq X_{\sigma+\tau}$, $\sigma,\tau\in \Lambda$.
If \mA\ possesses an involution, we further assume that
$ \one \in X_1$  and $ X_\sigma=X^*_\sigma$ for all $ \sigma \in \Lambda$.
The \emph{$\sigma$-th approximation error} of  $a \in \mA$ by $X_\sigma$ is
$  E_\sigma(a)=\inf_{x \in X_\sigma} \norm{a-x}_\mA$.
We define approximation spaces $\app p r  \mA$ by the  norm 
\begin{equation}
  \label{eq:appspace}
  \norm{a}_{\mE_r^p}^p= \sum_{k=0}^\infty {E_k(a)^p}(k+1)^{r p}\frac{1}{k+1},  \text{ for } \Lambda=\bn_0 \, ,
 \end{equation}
for $1\leq p < \infty $ with the  obvious change for $p=\infty$. If $\Lambda=\br_0^+$  an equivalent norm is 
 $
 \norm{a}_{\mE_r^p}^p=\int_0^\infty {E_\sigma(a)^p}(\sigma+1)^{r p}\frac{\dd \sigma}{\sigma+1} 
 $. 
Algebra properties of \as s of approximation spaces are discussed in~\cite{Almira06, grkl10}. In particular, in~\cite{grkl10} the following result is proved.
\begin{prop} \label{appspIC} If \mA\ is a symmetric \BA\ and $(X_\sigma)_{ \sigma \in \Lambda}$ an
  approximation scheme, then $\mE_r^p(\mA)$ is \IC\ in $\mA$.
\end{prop}
%The following charcterization of approximatio spaces is well known (see, e.g~\cite{Pietsch81}) 
 If \mA\ is a \MA\ and 
 $\mT_N=\mT_N(\mA)$ denotes the set of matrices in \mA\ with bandwidth smaller than $N$,
\[
\mT_N=\set{A \in \mA \colon A=\sum_{\abs k_\infty < N}\hat A (k)}
\]
then the sequence $(\mT_N)_{N\geq0}$ is an approximation scheme for \mA.
The closure of all banded matrices in \mA\ is the space of \emph{band-dominated matrices} in \mA
~\cite{Rabinovich98, rrs04}. 

In \cite{grkl10} we obtained the following constructive description of $\app \infty r{\mC^1_0}$:
\emph{The approximation space $\mE^\infty_r(\mC^1_0)$
 consists of all matrices $A$ satisfying
}  \begin{equation*}\label{eq:44}
    \norm{\hat A (0)}_{\bop}<\infty, \quad
   2^{rk}\smashoperator[l]{ \sum_{2^k \leq \abs l < 2^{k+1}}}\norm{\hat A
      (l)}_{\bop} = 2^{r k }\smashoperator[l]{\sum_{2^k \leq \abs l < 2^{k+1}}}  \sup _{m\in
      \bzd} |A(m,m-l)|  \leq C %\quad . 
  \end{equation*}
for all $k\geq 0$.
Theorem~\ref{prop:polappHMAs} is a more general result of this type.

\section{Smoothness in Banach Algebras}
\label{sec:smoothnessBA}
An important observation in~\cite{grkl10}  was that
the \odd\ of matrices can be described by \emph{smoothness properties}, using derivations and the action of the automorphism group 
$
\chi_t(A)=\sum \hat A(k) e^{2 \pi i k t}
$.
 In our treatment  we focused on \HZ\ spaces  and on spaces of $m$ times differentiable elements. Now we extend our research and cover the more general Besov and Bessel potential spaces. We also establish the isomorphism between Besov spaces and approximation spaces of polynomial order. In all cases we obtain results on the \IC ness of the smoothness spaces in their defining algebra.

It turns out that the investigations can be carried out with no additional effort for \BA s with the bounded action of a $d$-parameter automorphism group. Here we obtain new methods for the construction of scales of \IC\ subalgebras.
\subsection{Automorphism Groups and Continuity}
\label{Automorphism groups and continuity}
% Our next step is to treat the algebras $\mA_{v_r}$ with non-integer parameter $r$ in analogy to spaces with
% fractional smoothness.  Two natural approaches are fractional powers of the generators or automorphism groups and
% the associated H\"older-Zygmund continuity. We choose the latter approach and introduce a new structure, namely
% automorphism groups. This choice is also  motivated by the failure to distinguish between the spaces $\mD(\frac{d}{dx},
% L^\infty(\bT))=\set{ f \in \Lip(\bT): f' \in L^\infty(\bbT)}$ and $\mD(\frac{d}{dx}, C(\bbT))=C^1(\bbT)$ by means of
% derivations alone. To explain this difference, we need to consider derivations that are generators of groups of
% automorphisms.

Let \mA\ be a \BA. A ($d$-parameter) \emph{automorphism group}
 acting on \mA\   is a set of \BA\ automorphisms $\Psi=\set{\psi_t}_{t \in \brd}$ of
\mA\ that satisfy the group properties
\begin{equation*}
  \psi_s \psi_t=\psi_{s+t} \quad \text{for all} \quad s,t \in \br^d.
\end{equation*}
If \mA\ is a $*$-algebra we assume that $\Psi$ consists of $*$-automorphisms.  In order to simplify some proofs, we assume that
$\Psi$ is  \emph{uniformly bounded}, that is,
\[
M_\Psi= \sup_{t\in \brd}\norm{\psi_t}_{\mA \to \mA} < \infty \, .
\]
The abstract theory works for  more general group actions~\cite{Amrein96, grushka07, Torba08}. %\texttt{TORBA im arxiv: Wie zitiert?}

An element $a \in\mA$ is (strongly) \emph{\cont}, if
\begin{equation}
  \label{strongcont}
  \norm{\psi_t(a)-a}_\mA \to 0 \text{ for } t \to 0.
\end{equation}
The  \cont\ elements of \mA\ are denoted by $C(\mA)$.

For $t \in \brd\setminus\set{0}$ the \emph{generator} $\delta_t$ is
\begin{equation}\label{eq:generator}
  \delta_t (a) = \lim_{h \to 0} \frac{\psi_{ht}(a)-a }{h}
\end{equation}
The domain $\mD(\delta_t, \mA)$ of $\delta_t$ is the set of all $a \in \mA$ for which this limit exists. % Let
The \emph{canonical generators} of $\Psi$ are  $(\delta_{e_k})_{1\le k\le d}$, and
 $\Psi$ is  the \emph{automorphism group generated by} $(\delta_{e_k})_{1\le k\le d}$. If $\alpha \in \bnd_0$ is a multi-index, then $\delta^\alpha= \delta^{\alpha_1}_{e_1} \cdots \delta^{\alpha_d}_{e_d}$. In \cite[3.15]{grkl10} precise conditions are given, under which condition two derivations commute. In particular, this is true for all cases that will be encountered in this text.
%In this case 
Each generator $\delta_t$ is a closed derivation, that is
$\delta_t(ab)= a \delta_t(b) + \delta_t(a) b$ for all $a,b \in \mD(\delta_t,\mA)$, and the operator $\delta_t$ is a closed operator on its domain.
If \mA\ is a
Banach $*$-algebra, then $\delta_t$ is a $*$-derivation~\cite{bratrob87}. 
\begin{prop}[{\cite[3.4]{grkl10}}]
  If \mA\ is symmetric, and $\delta$ is a closed $*$-derivation then $\mD(\delta_t, \mA)$ is \IC\ in \mA.
\end{prop}
 We call the action of $\Psi$ periodic if $\psi_{t+e_k}=\psi_t$ for all $1 \leq k \leq d$ and all $t \in \brd$.% with period $P \geq 0$ or
%   \emph{P-periodic}, if there is a $P \in \brd_+$  such that $\psi_{t+P}=\psi_t$ for all $t \in \brd$.
% If we speak of a periodic group action we usually mean a 1-periodic group action ($P_k=1$ for all $k=1,\dotsc,d$).
% It is easy to check that for each $P$-periodic group action $\Psi$ we can define a new automorphism group $\overline{\Psi}$ by $\overline \psi_{(t_1, \dotsc, t_d)}= \psi_{(t_1/P_1, \dotsc, t_d/P_d)}$ that is $1$-periodic. The reader may verify that the smoothness spaces we will define in the sequel do not depend on this normalization.
%\begin{rems}
%  (1) In a $C^*$-algebra all automorphisms are isometries. This is no longer true for symmetric algebras.
% \end{ex}
\begin{defn}
A \MA\ \mA\ is called \emph{\hmg}~ \cite{Deleeuw75, Deleeuw77}, if 
    %$\chi = \set{ \chi _t}_{t \in \brd}$,
    \begin{equation*} \chi_t \colon A \mapsto M_t A M_{-t},\; \chi_t(A)(k,l)=\cexp [(k-l) \cdot t]A(k,l) \text{ for all } k,l \in
    \bzd \text{ and } t \in \brd
  \end{equation*}
   define  uniformly bounded mappings on \mA, where   $M_t, t\in \br^d$, is the\emph{ modulation operator} $M_t x (k)= \cexp [k \cdot t] x(k)$, $ k\in \bzd$.
\end{defn}
Clearly $\chi=\set{ \chi _t}_{t \in \brd}$ defines an automorphism group on \mA. The algebra of bounded operators on $\ell^2$ is a \HMA, and so are all solid \MA s. 

In the literature on group actions it is often assumed that $\Psi $ is strongly continuous on all of \mA,\
  i.e. $\mA = C(\mA )$.  This is no longer true for most  matrix algebras, and in general $C(\mA )$ is a closed and \IC\ subalgebra of \mA~\cite[3.14]{grkl10}. % e group

 In~\cite{grkl10} the spaces $C(\mA)$ have been identified for the algebras  $\mC^1_r, \mC^\infty_r, \mS^1_r$. We use the opportunity to state the full result for  the algebras $\mC^p_r$ and $\mS^p_r$.
%*** ZITIEREN ***?
 \begin{prop} %hspace{1cm} 
If $\mA$ is one of the algebras $\mC^\infty_r,\mS^p_r, \bopzd$ for $r \geq 0$, then $C(\mA) \neq \mA$.   

\[
C(\mC^p_r) =  \Ck [p]_r,  \quad 1 \leq p < \infty \,, 
\]

 \[   C(\Ck [\infty]_r) = 
 \set{A \in \Ck [\infty]_r \colon \lim _{\abs k_\infty\to \infty } \norm{\hat   A(k)}_{\Ck [\infty]_r} = \lim_{ \abs k_\infty \to \infty} \norm{\hat A (k)}_{\bop}(1+|k|)^r = 0 }
 \]

 \begin{align*}
   C(\schur [p] r)
   =\set{A \in \schur [p] r \colon  &\lim_{N \to \infty} \sup_{k\in \bzd} \sum_{\abs s_\infty > N} \abs{A(k,k-s)}^p (1+ \abs s)^{r p} =0 
     \text{ and } \\
     &\lim_{N \to \infty} \sup _{k\in\bzd} \sum_{\abs s_\infty > N} \abs{A(k-s,k)}^p(1+ \abs s)^{rp} =0}  \,.
 \end{align*}
\end{prop}
 The method of proof is as in~\cite{grkl10}.
\subsection{Besov Spaces}
    \label{sec:besov}
The theory of vector valued Besov spaces is well established~\cite{bergh76,Butzer67, Peetre76, Triebel92}. 
The main results of this section are the algebra properties of vector-valued Besov spaces derived from a given \BA\ \mA. Though possibly known, we were not able to find any references, so full proofs of the results are included.\\

Let \mA\ be a \BA\ with automorphism group  $\Psi$. Define the $k$th difference operator as
$\Delta_t^k=  (\psi_t  -\id)^k$,  $t \in \brd$. 
For a step size $h>0$,
the \emph{modulus of continuity} of $a \in \mA$ is $\omega_h(a)=\omega_h^1(a)=\sup_{\abs t \leq h} \norm{\Delta_t a}_\mA$. If $k > 1$, the $k$th \emph{modulus of smoothness} of $a$  is 
$
\omega^{k}_h (a,\mA) = \omega^{k}_h (a) = \sup_{\abs t \leq h} \norm{\Delta^k_t a}_\mA
$.
\begin{defn}%[Besov spaces]
  Let $1 \leq p\leq \infty$, $r >0$, $l=\floor r  +1$. The (vector valued) \emph{Besov space} $\Lambda^p_r(\mA)$ consists of all $a \in \mA$ for which the seminorm
\[
\abs{a}_{\Lambda^p_r(\mA)}= \begin{cases}
& \bigl(\int_{\br^+ } (h^{-r}\omega^l_h(a))^p \frac{dh }{h}\bigr)^{1/p} \,, \quad 1 \leq p < \infty \\
&\norm{a}_\mA +\sup_{h>0} h^{-r} \omega^l_h(a) \,, \quad p=\infty
\end{cases}
\]
is finite. The parameter $r$ is the smoothness parameter. The Besov norm is $\norm{a}_{\Lambda^p_r(\mA)}=\norm{a}_\mA +\abs{a}_{\Lambda^p_r(\mA)}$.
\end{defn}
Actually, replacing $l$ by any integer $k>\floor r$ in the preceding definition yields an equivalent norm for $\besov p r \mA$. In addition, we will need the following  norm equivalences.
 \begin{equation} \label{eq:besovnormeqs}
   \begin{split}
     \norm{a}_{\besov p r\mA} %&\asymp \norm{a}_\mA+ \Bigl(\int_{\br^+ } \big(h^{-r}\omega^k_h(a)\bigr)^p \frac{dh }{h} \Bigr)^{1/p} \\
     &\asymp \norm{a}_\mA + \Bigl(\int_{\brd} (\abs{t}^{-r}\norm{\Delta^k_t a}_\mA)^p \frac{dt }{\abs t ^d}\Bigr)^{1/p} \\
     &\asymp \norm{a}_\mA + \Bigl(\sum_{l=0}^\infty \bigl( 2^{r l} \omega_{2^{-l}}^k(a) \bigr)^p \Bigr)^{1/p} \,.
   \end{split}
\end{equation}
If $l \in \bn_0$, and $l \leq r$,  these norms are further equivalent to
\[
 \norm{a}_\mA+ \sum_{\abs \alpha =l} \norm{\delta^\alpha (a)}_{\besov p {r - l} \mA} \, .
\]
The Besov spaces  $\Lambda^p_r(\mA)$ are \BS s for all $1 \leq p\leq \infty$ and $r>0$. 
If $1 \leq p, q \leq \infty $ and $0 <r < s  $, then 
$\besov p s \mA \inject \besov q r \mA$.
If $p<q$ then $ \besov p r \mA \inject \besov q r \mA$.
See, e.g.,~\cite{bergh76, Butzer67,Peetre76,Triebel92} for these and other basic properties. 
\begin{lem}
  If $\mA$ is a \BA\ with (bounded) automorphism group $\Psi$,  then $\Psi$ is a  bounded automorphism group on  $\besov p r \mA$ for every  $1 \leq p \leq \infty$, and $r>0$.
\end{lem}
\begin{proof}
 Assume that $ p < \infty$. If $a \in \besov p r \mA$, $k > \floor r$ and $s \in \brd$, then for every $s>0$
 \begin{align*}
   \norm{\psi_s a}_{\besov p r \mA}
   = \norm{\psi_s a}_\mA + \Bigl( \int_{\brd} (\abs t ^{-r} \norm{\Delta^k_t \psi_s a}_\mA)^p \mulebd {t} \Bigr)^{1/p} 
   \leq M_\Psi \norm{a}_{\besov p r \mA} \,,
 \end{align*}
since $\delta^k_t\psi_s= \psi_s\delta^k_t$, so $\psi_s$ is bounded on $\besov p r \mA$. 
The proof for  $p=\infty$  is similar.
\end{proof}
%It follows from  that for 
\begin{prop}[{\cite[3.1.5,  3.4.3]{Butzer67}}] \label{prop:besovcontpart}
 If  $k > \floor{r}$ then
    \begin{align*}
    &C(\Lambda^p_r(\mA))=\Lambda^p_r(\mA), \quad  1\leq p < \infty \,, \\
    &C(\besov \infty r \mA)=  \lambda^\infty_r(\mA)=\set{a \in \mA \colon \lim_{h \to 0} h ^{-r} \omega^k_h(a) =0} \,.
  \end{align*}
\end{prop}
Does the iteration of the construction of Besov spaces yield refined smoothness spaces? %Fortunately, this is not the case:
\begin{thm}[Reiteration theorem]\label{thm-reiteration}
If $1 \leq p,q \leq \infty$ and $r,s >0$ then
  \begin{equation}\label{eq:1}
    \Lambda^q_s(\Lambda^p_r(\mA))=\Lambda^q_{r+s}(\mA) \,.
  \end{equation}
\end{thm}
A proof of is in appendix~\ref{sec:proof-reit-theor}.
\begin{rems}
A proof of the reiteration formula for the \BA\  of bounded operators on a  \BS\ $\mX$ and the automorphism group $\psi$ obtained by conjugation with an automorphism group on \mX\  has been given in\cite{Amrein96}, using interpolation theory.

We think the reiteration formula is of some conceptual interest. 
Note that the classical notion of Besov spaces on \brd\  does not even allow to formulate the result.  
 We  use \eqref{eq:1} to simplify proofs of approximation results.%, so we go for a  proof of the theorem.
\end{rems}
% Clearly it does not have a ``classical'' counterpart.
%
%  \subsection{Algebra Properties}
   % \label{sec:besov_algebra}
The main result of this section treats the algebra properties of Besov spaces.
\begin{thm} \label{thm-besov-IC}
   Let \mA\ be a \BA\ with automorphism group $\Psi$. For all parameters   $1 \leq p \leq \infty$ and $r>0$, 
      the Besov space $\besov p r \mA$ is a Banach subalgebra of \mA.
   Moreover, $\besov p r \mA$ is \IC\ in \mA.
 \end{thm}
 \begin{proof}
We treat the case $r <1$ first.
  To show that $\besov p r \mA$ is a  \BA\ we use the  identity
  \begin{equation} \label{eq:7}
    \Delta_t(ab)=\psi_t(a)\Delta_t (b)+\Delta_t(a)b \,.
   \end{equation}
Taking norms we obtain
\begin{equation*}
  \begin{split}
    \norm{\Delta_t(ab)}_\mA &\leq \norm{\psi_t(a)}_\mA \norm{\Delta_t(b)}_\mA +\norm{\Delta_t(a)}_\mA\norm{b}_\mA\\
               & \leq  M_\Psi\norm{a}_\mA \norm{\Delta_t(b)}_\mA +\norm{\Delta_t(a)}_\mA\norm{b}_\mA \,.
              % & \leq M_\Psi\bigl(\norm{a}_\mA \norm{\Delta_t(b)}_\mA +\norm{\Delta_t(a)}_\mA\norm{b}_\mA \bigr)  \,,
  \end{split}
\end{equation*}
%where we used  $M_\Psi \geq 1$ for the last inequality.
This implies a similar relation for the  Besov-seminorms, namely,
\begin{equation*}
  \abs{ab}_{\besov p r \mA}\leq M_\Psi \norm{a}_\mA \abs{b}_{\besov p r \mA}+\norm{b}_\mA \abs{a}_{\besov p r \mA}.
\end{equation*}
So
\begin{equation*}
\norm{ab}_{\besov p r \mA}=\norm{ab}_\mA+\abs{ab}_{\besov p r \mA} \leq %\norm{a}_\mA \norm{b}_\mA + 2 M_\Psi (\norm{a}_{\besov p r \mA}\norm{b}_{\besov p r \mA})=
 C \norm{a}_{\besov p r \mA}\norm{b}_{\besov p r \mA} \,,
\end{equation*}
and the assertion follows.

To show that   $\besov p r \mA$ is \IC\ in \mA\ we assume that $a \in \besov p r \mA$ is invertible in \mA. It is sufficient to verify that $\abs{a^{-1}}_{\besov p r \mA}$ is finite.
By a straightforward computation we obtain
\begin{equation} \label{eq:invderiv}
  \Delta_t(a^{-1})=-\psi_t(a^{-1}) \;  \Delta_t(a) \; a^{-1} \,.
\end{equation}
This implies that $a^{-1}$ has a finite $\besov p r \mA$-norm.
% As  $\lambda_\sigma=C(\besov p r \mA)$ the algebra properties and \IC ness of $\lambda_\sigma$ follow from  Proposition~\ref{prop:autom-groups-cont-IC}.

In the general case we can use the reiteration theorem (Theorem~\ref{thm-reiteration}) and the transitivity  of \IC ness, and prove the statement by induction. Assume that the statement is proved for all smoothness parameters smaller than $s>0$. 
As $\besov p r \mA = \besov p {r -s} {\besov p s \mA}$ for $r>s$, the preceding argument yields
$\besov p {r - s} {\besov p s \mA}$ is \IC\ in $\besov p s \mA$ for $s < r < s+1$. As $\besov p s \mA$    is \IC\ in \mA\ by hypotheses, the theorem is proved.
\end{proof}
\subsection{ Characterization of Besov Spaces as Approximation Spaces}
\label{sec:char-besov-spac}
As in the case of function spaces, $\besov p r \mA$ can be characterized by approximation properties. This was carried out for $\besov \infty r \mA$ in~\cite{grkl10}, so our treatment is very brief.

We focus on the approximation of a \BA\ \mA\ with automorphism group $\Psi$ by smooth elements. 
\begin{defn}[Bernstein inequality]
  An element $a \in \mA$ is \emph{$\sigma$-bandlimited} for $\sigma >0$, if there is a constant $C$ such that for every multi-index $\alpha$
  \begin{equation}
    \label{eq:20}
    \norm{\delta^\alpha( a)}_\mA \leq C ( 2 \pi \sigma)^{\abs \alpha} \,.
  \end{equation}
  An element is \emph{bandlimited}, if it is $\sigma$-bandlimited for
  some $\sigma
  >0$. % The set of all $\sigma$-\BL\ elements of \mA\ will be denoted as $\mA_\sigma$; let $\mA_\Sigma=\cup_{\sigma>0}\mA_\sigma$ be the set of all \BL\ elements of \mA.
  %Inequality \eqref{eq:20} is a generalized Bernstein inequality.
\end{defn}
If \mA\ is a \BA\ with automorphism group $\Psi$, then,
\begin{equation*}
      \label{eq:29}
      X_0=\set{0}, \quad X_\sigma=\set{a \in \mA \colon a \text{ is } \sigma \text{-\BL} }, \quad \sigma>0 
    \end{equation*}
is an approximation scheme for \mA\ \cite[Lemma 5.8]{grkl10}. From now on we  use this approximation scheme without further notice. 

In particular, if \mA\ is a\HMA, we obtain the following characterization of \BL\ elements
\begin{prop}[{\cite[5.7]{grkl10}}]
  A matrix $A$ is banded with bandwidth $N$ in the \HMA\ \mA, if and only if it is $N$-bandlimited with
  respect to the group action $\{\chi _t\}$.
\end{prop}

 \begin{thm}[Jackson Bernstein Theorem]\label{prop:jacksonbernstein}
    Let $\mA$ be a \BA\  with automorphism group $\Psi$, and assume that $r >0$ and $1 \leq p \leq \infty$. If $\{X_\sigma : \sigma \geq 0\}$ is the approximation
    scheme of bandlimited elements, then
    \begin{equation}
      \label{eq:cha1}
      \besov p r \mA  = \mE ^p _r (\mA ) \, .
    \end{equation}
  %  In other words, $a \in \Lambda_r(\mA)$, if and only if $ E_\sigma(a) \leq C \sigma^{-r}$ for all $ \sigma>0$.
  \end{thm}
The proof is in appendix \ref{sec:jacks-bernst-theor}. 
%Alternatively this result can be obtained from the one in \cite{grkl10} using an interpolation argument\cite[7.9.2]{DeVore93} ***nicht genau, eher Pietsch !! ***

\subsubsection*{Littlewood-Paley Decomposition}
\label{sec:littl-paley-decomp}
The norms of Besov spaces  are not easily computable. An equivalent explicit norm  for these spaces can be obtained by means of a \LP\ decomposition.

First we need some technical preparation:
  If $\mu \in \mM(\brd)$ and $a \in C(\mA)$, the \emph{action of} $\mu$ on $a$ is defined by
   \begin{equation}\label{eq:module}
    \mu * a = \int_{\brd}\psi_{-t}(a) d\mu(t). 
  \end{equation}
  This action is a generalization of the usual convolution and satisfies similar properties: 
\[ \norm{\mu * a}_\mA \leq M_\Psi  \norm{\mu}_{M(\brd)}\, \norm{a}_\mA  \,. \]
If $f \in C_c^\infty(\brd)$  then 
\begin{equation}\label{eq:derivderiv}
\delta^\alpha(f * a)=D^\alpha f *a \in C(\mA)
\end{equation}
for every multi-index $\alpha$.
See
  {\cite{Butzer67}} for details and
  proofs.% $C(\mA)$ is a \emph{convolution module:} (in den language of\textsc{Feichtinger?})

In particular, if the group action is periodic,
the action of $\mu$ on $a$ is
\begin{equation}
  \label{eq:16}
  \mu * a = \int_{\btd} \psi_{-t}(a) \, d\mu(t)=  \sum_{k \in \bzd} \mF(\mu)(k) \hat a(k)\,, 
\end{equation}
where  $\hat a (k)=\int_{\btd} \psi_{-t}(a) \cexp[ {k \cdot t}] \,dt $ is the $k$-th Fourier coefficient of $a$ and the sum converges in  the C1-sense.% as in Proposition~\ref{prop_DeLeeuw}.
\\
%
%\subsection*{Dyadic scaling functions}\label{sec:dyad-scal-funct}
Now assume that $\varphi \in \schwartz(\brd)$ with
  $\supp \hat \varphi \subseteq \set{\omega \in \brd \colon 2^{-1} \leq \abs{\omega}_\infty \leq 2}$,
  $\hat \varphi(\omega) >0$  for    $2^{-1} <\abs{\omega}_\infty <2$, and
  $\sum_{k \in \bz}\hat \varphi(2^{-k}\omega)=1$ for all $\omega \in \brd\setminus \set{0}$.
 Set
$\hat \varphi_{k}(\omega)=\hat \varphi(2^{-k}\omega), k\in \bn_0$, so $ \varphi_k(x)= 2^{kd} \varphi_0(2^k x)$, and   let $\hat  \varphi_{-1}=1-\sum_{k=0}^\infty \hat\varphi_k$.
 Then $\set{\hat \varphi_k}_{k \geq -1}$ is a dyadic partition of unity. %The existence of  \DPU s is well kown.

\begin{prop} \label{prop:littl-paley-decomp-3} Let $\set{\hat\varphi_k}_{k\geq -1}$ be a dyadic partition of
  unity,  and $ 1\leq p \leq \infty$, $r>0$.  An element $a \in \mA$ is in $\besov p r \mA$,  if and only if
  \begin{equation}
    \label{eq:39}
     \biggl( \sum_{k = -1}^\infty 2^{r k p}\norm{\varphi_k * a}_\mA^p \biggr)^{1/p} < \infty \,.
  \end{equation}
   The expression \eqref{eq:39} defines an equivalent norm on $\besov p r \mA$.
Moreover the \LP\ decomposition $\sum_{k=0}^\infty\varphi_k*a $  converges to $a$ in the norm of \mA.
\end{prop}
 The special case $p=\infty$ was proved in~\cite{grkl10} with a weak type argument. This approach does not work for $ p < \infty$, so we adapt a proof  in~\cite{bergh76}, see Appendix~\ref{sec:lp-decomposition-1}.

\subsubsection*{Approximation of Polynomial Order in Homogeneous Matrix Spaces}
\label{sec:hz-spaces-ma}
\begin{lem}
  If $\set{\varphi_k}_{k\geq -1}$ is a \DPU\ and the action of $\Psi$
  on \mA\ is periodic, then for $a \in C(\mA)$,
  \begin{equation}
    \label{eq:periodicbesov}
    \varphi_k*a = \sum_{\floor{ 2^{k-1}} \leq \abs l_\infty <2^{k+1}} \hat \varphi_k(l) \hat a(l)  \,.
  \end{equation}
\end{lem}

\begin{proof}
  Let $\varphi_k^\Pi(t) =\sum_{l \in \bzd}\varphi_k(t +l)$
  denote the periodization of $\varphi_k$. Then
\[
\varphi_k^\Pi(t) =\sum_{\floor{ 2^{k-1}} \leq \abs l_\infty <2^{k+1}}\hat \varphi_k(l) \cexp[ l\cdot t]
\]
by Poisson's summation formula. Equation~\eqref{eq:periodicbesov} now follows, combining~\eqref{eq:16} with
\[
\varphi_k*a =\int_{\brd}\psi_{-t}(a) \varphi_k(t) \, dt=\int_{\btd}\psi_{-t}(a)\varphi^\Pi_k(t) \, dt  \qedhere
\]
\end{proof}
Equation~\eqref{eq:periodicbesov} allows us to obtain a  characterization of the \as s for \HMA s by the \LP\ decomposition of its elements.
\begin{prop}
  \label{prop:polappHMAs}
Let $\mA$ be a \HMA, $r>0$, and $\Phi=\set{\varphi_k}_{k\geq-1}$ a \DPU. Then the norm on the \as\ $\app p r \mA = \besov p r \mA$ is equivalent to
\begin{equation}
  \label{eq:lpnormonas}
  \norm{A}_{\app p r \mA} \asymp
      \biggl( \sum_{k=0}^\infty 2^{k p r} \bignorm{\sum_{\floor{2^{k-1}} \leq \abs l_\infty <{2^{k+1}}} \hat\varphi_k(l)\hat A (l)}_\mA^p \biggr)^{1/p} \,. 
\end{equation}
%If $v^*_s(k)= 1+ (2 \pi \abs k)^2)^{s/2}$, $s>0$ is a Bessel weight, then 
% \[
% \app p r {\mA_{v^*_s}} = \app p {r +s} \mA \, .
% \]
%The spaces $\app p r \mA$ and $\mA_{v^*_s}$ are \IC\ subalgebras of \mA.
%
If $\mA$ is solid, then %the statements simplify to
\begin{equation}
  \label{eq:lpnormonassolid}
  \norm{A}_{\app p r \mA} \asymp
      \biggl( \sum_{k=-1}^\infty 2^{k p r} \bignorm{\sum_{\floor{2^{k}} \leq \abs l_\infty <{2^{k+1}}} \hat A (l)}_\mA^p \biggr)^{1/p} \,. 
\end{equation}
% and
% \begin{equation}\label{eq:appweighted}
% \app p r {\mA_{v_s}} = \app p {r +s} \mA \, .
% \end{equation}
\end{prop}
\begin{rem}
  If the \MA\ \mA\ is solid, similar results can be obtained for approximation spaces of non-polynomial order~\cite{Klo09}.
\end{rem}
\begin{proof}
  The results for general \HMA s follow from the Jackson Bernstein Theorem (Theorem~\ref{prop:jacksonbernstein}) and  the  \LP\ decomposition. We still have to prove the norm equivalence~\eqref{eq:lpnormonassolid}.
Set $C_k =  \norm{\sum_{2^{k}\leq |l| <2^{k+1}} \hat A(l)}_\mA $.
 The solidity of \mA\  implies that, for $k\geq -1$,
  \[
  B_k = \norm{ \varphi _k * A}_{\mA }
% = \norm{\sum_{l \in \bzd} \phi_k(l) \hat A(l)}_\mA 
\leq \norm{\sum_{2^{k-1} \leq \abs l_\infty < 2^{k+1}} \hat A(l)}_\mA =C_{k-1} + C_k
  \]
  On the other hand, since $\phi_{k-1}+\phi_k + \phi_{k+1} \equiv 1$ on $\set{\xi \colon 2^{k-1} \leq \abs \xi _2
    \leq 2^{k+1}}$, we obtain $C_k \leq B_{k-1}+B_k+B_{k+1}$. So
$\norm{A}_{\app p r \mA}^p \asymp \sum_{k=0}^\infty2^{kpr}B_k \asymp \sum_{k=-1}^\infty2^{kpr}C_k$,
and this is (\ref{eq:lpnormonassolid}).
%
%The identity (\ref{eq:appweighted}) follows from Proposition~\ref{prop:besselbesovMA} together with  $\mA_{v_s}=\mA_{v^*_s}$, which is valid in solid \MA s.
\end{proof}
We  apply the preceding results to the example of  $\mC^p_r$ (see Section~\ref{sec:exma} for the definition).
We obtain
\begin{align*}
  \label{eq:appbaskjaff1}
&\app q s {\mC^p_r} = \app q {s+r} {\mC^p_0} \,, \\
%\norm{a}_{\app q s {\mC^p_r}} \asymp 
&\norm{A}_{\app q r {\mC^p_s}}
%&\asymp \Bigl( \sum_{j=0}^\infty 2^{jqr}\bigl({\sum_{\floor{2^{\j-1}} \leq \abs k_\infty < 2^j}}\norm{ A [k]}_{\lp [\infty]}^p(1+\abs k)^{s p} \bigr)^{q/p} \Bigr)^{1/q} \\
\asymp \Bigl( \sum_{j=0}^\infty 2^{jq(r+s)}\bigl({\sum_{\floor{2^{\j-1}} \leq \abs k_\infty < 2^j}}\norm{ A [k]}_{\lp [\infty]}^p \bigr)^{q/p} \Bigr)^{1/q} 
\asymp \norm{A}_{\app q {r+s} {\mC^p_0}}\,.
\end{align*}
In particular,
\[
\app  p s {\mC^p_r} =\mC^p_{r+s} \,.
\]
If $p \neq q$ these norms  define new classes of \IC\ subalgebras of \bop\ with a form of \odd\ suited to approximation with banded matrices.

These results should be compared to the definition of \emph{discrete Besov spaces}~\cite{pietsch80}.

\subsection{Bessel Potential Spaces}
\label{sec:bess-potent-spac}
Bessel potentials allow us to define an analogue of \emph{polynomial} weights in a \BA\  with an automorphism group. For \HMA s the Bessel potential spaces are weighted algebras. 

We define the Bessel kernel $\mG_r$ by its Fourier transform,
\[
\mF \mG_r( \omega)= (1+\abs{2 \pi \omega}_2^2)^{-r/2} \, , \quad r >0.
\]

\subsubsection*{$C_w$ groups}
\label{sec:defin-basic-prop}
In analogy to the case of real functions we would like to define an element of the Bessel potential space $\mP_r(\mA)$  as an element of the form $a= \mG_r * y$ for some $y \in  \mA$.
 However, the action $\mG_r * y$ is defined only for $y \in C(\mA)$. Using  a weaker form of continuity for the action of the automorphism group we can extend the convolution ``$*$'' to the whole algebra for all  examples of \MA s in Section~\ref{sec:exma}.  
%*******************
 \begin{defn}[\cite{Arveson82, bratrob87}] \label{defn- cwgroup}
   % The automorphism group $\Psi$ is a $C_0^*$ group, if $\mA =\mX'$ for a \BS\ $\mX$, and  In this case the predual $\mX$ of $\mA$ will be denoted as $\mA_*$.
   Let $\mA$ be a \BA\ with automorphism group $\Psi$. For $a \in \mA$, $a' \in \mA'$  define $G_{a',a}(t)=\inprod{a', \psi_t(a)}$. Assume that
   % $a' \in \mX$, $a \in \mA$.More generally, let us assume that
   \[
   \mA'_\Psi = \set{ a' \in \mA' \colon G_{a',a} \text{ is \cont\ for all } a \in \mA }
   \]
   is a \emph{norm fundamental subspace} of $\mA'$, that is
   \[
   \norm{a}_\mA = \sup \set{ \abs{\inprod{a',a}} \colon a' \in \mA'_\Psi, \, \norm{a'}_{\mA'} \leq 1}
   \]
   for all $a \in \mA$. Assume that \mA\  is equipped with the weak topology  $\sigma(\mA,\mA'_\Psi)$ with respect to the functionals in $\mA'_\Psi$, and
\begin{equation}\label{eq:convomp}
\text{\emph{the convex hull of every }} \sigma(\mA,\mA'_\Psi) \text{\emph{-compact set has} } \sigma(\mA,\mA'_\Psi) \text{\emph{-compact closure}.}
\end{equation}
In this case we call $\Psi$ a $C_w$
   group and denote $\mA$ with the $\sigma(\mA,\mA'_\Psi)$ topology by $\mA_w$, if necessary. %.\footnote{ the condition (\eqref{eq:convomp}) is not mentioned in \cite{Amrein96}}
 \end{defn} 
Condition \eqref{eq:convomp} ensures the existence of the ``convolution integral'' \eqref{eq:module} as a Pettis integral, see below. If  $\sigma(\mA,\mA'_\Psi)$ is \emph{quasi-complete}, i.e., bounded Cauchy nets converge, then condition \eqref{eq:convomp} is automatically satisfied \cite{koethe69}.
\begin{ex}\label{normfund}$\phantom{x}$
  \begin{enumerate}% \hspace {-1cm}
  \item \label{predual}If $\mA'_\Psi$ is the
    predual of \mA\ (in particular, if \mA\ is a von Neumann algebra)
    the quasi-completeness is a consequence of the Banach-Alaoglu
    theorem.

  \item \label{koethedual}If $\mA$ is a Banach function space in the
    sense of~\cite{bennett88}, and $\mA'_\Psi$ is a norm fundamental
    order ideal of the \emph{Koethe dual} $\mA^\sim$, then it is known
    that $(\mA, \sigma(\mA, \mA'_\Psi))$ is
    quasi-complete~\cite[1.5.2]{bennett88}.

  \item \label{normdual} If $\mA = C(\mA)$ then $\mA'_\Psi =
    \mA'$. It is well-known that $\psi_t$ is strongly \cont\ at $a \in
    \mA$, if and only if it is \cont\ with respect to the $\sigma(\mA,
    \mA')$-topology \cite{Butzer67,hille57}, so in this case $\mA_w=
    \mA$. In this case the condition \eqref{eq:convomp} is a consequence of the Krein-Smulian theorem.
  \end{enumerate}
\end{ex}
 \begin{rems}
   If the group action is uniformly bounded the space $\mA'_\Psi$ is a norm-closed subspace of $\mA'$. Indeed, if $a_k' \in \mA'_\Psi$ and $a_k' \to a'$ in norm, then %In particular, if $\mA'_\Psi = \mA_*$ we are back at the concept of a $C^*_0$ group, the quasi-completeness is a consequence of the Banach-Alaoglu theorem. 
\begin{align*}
\lim_{t \to 0} G_{a',a}(t)-G_{a',a}(0) 
&=\lim_{t \to 0} \inprod{a', \psi_t(a)-a}
=\lim_{t \to 0} \lim_k \inprod{a_k', \psi_t(a)-a} \\
&=\lim_k \lim_{t \to 0} \inprod{a_k', \psi_t(a)-a}
=0 \,.
\end{align*}
 \end{rems}

We do not have general conditions when the action of $\chi$ on a \HMA\ is a $C_w$-group. For the specific examples of \MA s introduced in Section~\ref{sec:exma} we can prove that $\chi$ is a $C_w$ group.
\begin{ex} 
%In the following we verify that the  action of $\chi$ on the standard matrix algebras yields  $C_w$ groups. 
Recall that \bop\ is the dual of the trace class operators  $\mB_1$, $\bop=(\mB_1)'$ and the finite rank  operators are dense in $\mB_1$. Adapting a  continuity argument from~\cite{Deleeuw75} we verify that $(\bop)'_\chi \supseteq \mB_1$. Indeed, for $x,y \in \ltwo(\bzd)$ and the rank one operator $(x \otimes y) z= \inprod{z,y} x$ we obtain
\begin{align*}
  G_{x \otimes y, A} (t)- G_{x \otimes y, A} (t) =& \operatorname{tr}((x \otimes y)  \chi_t(A))-\operatorname{tr}((x \otimes y)  A)= \inprod{x, (\chi_t(A)-A)y}\\ 
  =& \inprod{x, M_tAM_{-t}(y-M_ty)} %+\inprod{x,M_t(Ay-M_{-t}Ay)}
\end{align*}
As  $\lim_{t \to 0}\norm{z-M_tz}_{\ltwo(\bzd)} = 0$  for every $z \in \ltwo (\bzd)$, it follows that $G_{x \otimes y, A}$ is \cont. So, if $A'$ is a finite  rank operator then $G_{A',A}$ is \cont. As $\mA'_\Psi$ is norm closed in $\mA'$, and the finite rank operators are dense in $\mB_1$ we obtain the continuity of $G_{A',A}$ for all  $A' \in \mB_1$. We have shown that the space  $\mA'_\Psi$ contains $\mB_1$. This implies that $\mA'_\Psi$  is  norm fundamental, and so $\chi$ is a $C_w$-group on $\bop$. %(***is $\mA_*$ norm closed in $\mA'$ ***?) Assume that $A_k$ of finite rank and    explicitely
\end{ex}
\begin{ex}
If $\mA = \mS^p_r$ we can argue as follows: Let $\ell^{\infty,p}_{m_r}(\bzd [2d])$ the  mixed norm space on $\bzd [2d]$ with
\[
\norm{ (x(k,l))_{k,l \in \bzd}}_{\ell^{\infty,p}_{m_r}}
= \sup_{k \in \bzd} (\sum_{l \in \bzd} \abs{x(k,l)}^p (1+\abs{k-l})^{rp})^{1/p}
\]
and define $(j x) (k,l)= x(l,k)$. Then we obtain  the isometric isomorphism
\[
\mS^p_r \cong \ell^{\infty,p}_{m_r}(\bzd [2d]) \cap j \bigl( \ell^{\infty,p}_{m_r}(\bzd [2d]) \bigr) \,.
\]
 From \cite[Lemma 1.12]{cwikel03} (and using standard facts about sequence spaces, e.g \cite[30.3]{koethe69} we conclude that
\[
\bigl(\ell^{\infty,p}_{m_r}(\bzd [2d]) \cap j \bigl( \ell^{\infty,p}_{m_r}(\bzd [2d]) \bigr) \bigr)^\sim
\cong
\ell^{1,p'}_{m_{-r}}(\bzd [2d]) + j \bigl( \ell^{1,p'}_{m_{-r}}(\bzd [2d]) \bigr)
\]
It is routine to verify that $G_{A',A}$ is \cont\ for $A' \in \ell^{1,p'}_{m_{-r}}(\bzd [2d]) + j \bigl( \ell^{1,p'}_{m_{-r}}(\bzd [2d]) \bigr)$ and $A \in \mS^p_r$, so
Example \ref{normfund} (\ref{koethedual}) verifies that $\mS^p_r$ is a $C_w$-group.
\end{ex}
We need the concept  of $C_w$-groups not only to extend the action of a measure defined in \eqref{eq:module} to the whole of \mA, but also to give a weak type description of this action.
\begin{prop}[{\cite[1.2]{Arveson74}}]
 If $\Psi$ is a $C_w$- group for the \BA\ \mA, then for each $\mu \in \mM(\brd)$ and each $a \in \mA$ there is an element, denoted as  $  \mu * a \in \mA$, such that 
 \[
   \inprod{a', \mu *a} = \int_{\brd} \inprod{a', \psi_{-t}(a)} \, d\mu (t)
  \]
for all $a' \in \mA'_\Psi$ . As usual we write 
\begin{equation} \label{eq:weakmodule}
   \mu *a = \int_{\brd} \psi_{-t}(a) \, d\mu (t) \,.
 \end{equation}
 We obtain the norm inequality
 \[
 \norm{\mu *a} \leq M_\Psi \norm{a}_\mA \norm{\mu}_{\mM(\brd)} \,.
 \]
\end{prop}
\begin{rems}
  Clearly in special cases the existence of the integral \eqref{eq:weakmodule} can be verified directly. In particular, if $\mA= C(\mA)$ the integral exists in the sense of Bochner.
\end{rems}
The following result is straightforward.
\begin{prop}
  If $\Psi$ is a $C_w$- group for the \BA\ \mA, and $G_{a',a}(t)= \inprod{a', \psi_t(a)}$ for $a' \in \mA'_\Psi$, $a \in \mA$, then
\[ 
\norm{a}_\mA \asymp \sup \set{\norm{G_{a',a}}_\infty\colon a' \in \mA'_\Psi , \norm{a'}_{\mA'} \leq 1}  \,.
\]
Moreover, %the following relation holds
\begin{equation}
  \label{eq:convcalc}
  G_{a', \mu * a}= \mu * G_{a',a} \,.
\end{equation}
\end{prop}
%********************************************************************
Before defining Bessel potential spaces we list properties of the Bessel kernel that will be needed in the sequel.
\begin{lem}[{\cite[V.5]{stein70}}]\label{lem-besselkern-besov} \hspace{1cm}
 \begin{enumerate} 
\item    $\mG_r \in \besov \infty r {L^1(\brd)}$ , $\norm{\mG_r}_{L^1(\brd)}=1$,
\item    $\mG_r *\mG_s = \mG_{r + s}$ for all $r, s > 0$,
\item \label{surschwartz}   $\mG_r*\schwartz =\set{\mG_r*\varphi \colon \varphi \in \schwartz} =\schwartz$.
  \end{enumerate}
\end{lem}
\begin{defn}\label{defn:bessel}
  Let \mA\ be a \BS\ and $\Psi$ a $C_w$- group acting on
  \mA\ (this includes the case $A= C(\mA)$). The \emph{Bessel potential space} of order $r>0$ is
  \[
  \mP_r(\mA)= \mG_r * \mA = \set{a \in \mA \colon a= \mG_r * y \text{ for some } y \in \mA }
  \]
  with the norm
  \[
  \norm{\mG_r * y}_{\bessel r \mA} =\norm{y}_\mA.
  \]
\end{defn}
We have to verify that the definition of the norm  on ${\bessel r \mA}$ is consistent, that is, we show that 
the convolution with $\mG_r$ is injective on \mA. We use a weak type argument.

Let $y \in \mA$ with  $\mG_r*y =0$. This is equivalent to
\[
G_{{a',\mG_r *y}}(t)=\mG_r *G_{{a',y}}(t) = 0
\]
for all $t \in \brd$ and all $a' \in \mA'_\Psi$. Now we proceed as in ~\cite[V.3.3]{stein70}. We choose a test function $\varphi \in \schwartz$ and obtain 
\[
\int_{\brd}(\mG_r *G_{{a',y}}) (t)  \varphi(t) \, dt =\int_{\brd} G_{{a',y}}(t) (\mG_r *\varphi) (t)  \, dt  =0\,.
\]
By Lemma~\ref{lem-besselkern-besov} (\ref{surschwartz}) the convolution with $\mG_r$ is surjective on \schwartz, and so it follows that 
$G_{{a',y}}=0$ for all $a' \in \mA'_\Psi$, that is, $y=0$. \\

An immediate consequence of Definition~\ref{defn:bessel} is the embedding $\bessel r \mA \inject \mA$.
Indeed, if $a \in \bessel r \mA$, then $a= \mG_r*y$ for a $y \in \mA$, and
\begin{equation} \label{eq:besselinX}
  \norm{a}_\mA \leq \norm{\mG_r}_{L^1(\brd)} \norm{y}_\mA =\norm{\mG_r}_{L^1(\brd)} \norm{a}_{\bessel r \mA} .
\end{equation}

As $\mG_r*\mG_s=\mG_{r+s}$ for $r,s>0$  we obtain a useful reiteration property for the Bessel potential spaces.
\begin{prop} \label{reitbessel}
If \mA\ is a \BA\ and $\Psi$ a $C_w$-automorphism group on \mA, then for all $r,s >0$
  \[ \mP_r(\mP_s(\mA)) = \mP_{r+s}(\mA) \,.  \]
\end{prop}

\begin{proof}
  We have to verify that $\Psi$ is a $C_w$-automorphism group on
  $\bessel r \mA$. For this we show that the dual pairing defined by
  \[
  \inprod{a',\mG_r*y}_{\mA'_\Psi \times \bessel r \mA}
  =\inprod{a',y}_{\mA' \times \mA}
  \]
  yields a norm-fundamental subspace of $\bessel r \mA '$.  As
  \[
  \abs{\inprod{a', \mG_r*y}_{\mA'_\Psi \times \bessel r \mA}} \leq
  \norm{a'}_{\mA'}\norm{y}_\mA=\norm{a'}_{\mA'}\norm{\mG_r*y}_{\bessel
    r\mA}
  \]
  the mapping $z \mapsto \inprod{a', z}_{\mA'_\Psi \times \bessel r
    \mA}$ is \cont\ for every $a' \in \mA'_\Psi$, so $\bessel r \mA
  '_\Psi \supset \mA'_\Psi$. Moreover, a straightforward computation
  shows that $\norm{a'}_{\bessel r \mA '}=\norm{a'}_{\mA'}$.  By
  definition $t \mapsto \inprod{a', \psi_tz}_{\mA'_\Psi \times \bessel
    r \mA}$ is \cont\ for each $a' \in \mA'_\Psi$ and each $z \in
  \bessel r \mA$.  Finally, $\mA'_\Psi$ is norm fundamental, as we
  have for $z = \mG_r*y$
  \begin{align*}
    &\sup\set{ \abs{\inprod{a',y}_{\mA'_\Psi \times \bessel r \mA}} \colon a' \in \mA'_\Psi, \norm{a'}_{\bessel r \mA '}\leq 1 }\\
    =&\sup\set{ \abs{\inprod{a',y}_{\mA' \times \mA}} \colon a' \in
      \mA'_\Psi, \norm{a'}_{\mA '}\leq 1 } \\% , \no ***problem
    =& \norm{y}_\mA =\norm{z}_{\bessel r \mA}
  \end{align*}
\end{proof}

\subsubsection{Characterization by Hypersingular Integrals}
\label{sec:char-hypers-integr}
\begin{lem} \label{lem:weak-bessel-norm}
  If $a \in \bessel r \mA$, then
$
\norm{a}_{\bessel r \mA} \asymp \sup_{\norm{a'}_{\mA'} \leq 1} \norm{G_{a',a}}_{\bessel r {L^\infty}},
$
where the dual pairing in $G_{a',a}$ is the one of $\mA'_\Psi \times \mA$.
\end{lem}
% \begin{rem}
%   For $x' \in \mA'$ the expression $G_{x',a}$ is well-defined, as $\mA' \subseteq (\bessel r \mA)'$.
% \end{rem}
\begin{proof}
  Let $a=\mG_r*y$. Then 
  \begin{align*}
    \norm{a}_{\bessel r \mA} = &\norm{y}_\mA \asymp \sup_{\norm{a'}_{\mA'}\leq1} \norm{G_{a',y}}_\infty\\
    =&\sup_{\norm{a'}_{\mA'}\leq1} \norm{\mG_r*G_{a',y}}_{\bessel r  {L^\infty}} = \sup_{\norm{a'}_{\mA'}\leq1}
    \norm{G_{a',\mG_r*y}}_{\bessel r {L^\infty}} \,. \qedhere
  \end{align*}
\end{proof}
We state a special case of a result by Wheeden~\cite{Wheeden68} (see also~\cite{Stein61},\cite[V.6.10]{stein70}).
\begin{thm}\label{thm-hypersing-1d}
  Let $0 < r < 2$. A function $f$ is an element of $\bessel r {L^\infty(\brd)}$ if and only if $f \in L^\infty(\brd)$ and 
  \begin{equation}
    \label{eq:wheeden}
    \sup_{\epsilon > 0} \bignorm{\int_{\abs t_2 \geq \epsilon} \abs{t}_2^{-r} \Delta_t(f) \mulebdii{t}}_{L^\infty(\brd)} <\infty .
  \end{equation}
If \eqref{eq:wheeden} holds, 
\begin{equation}
  \label{eq:singintnorm}
\norm{f}_{L^\infty(\brd)}+\sup_{\epsilon > 0} \bignorm{\int_{\abs t_2 \geq \epsilon} \abs{t}_2^{-r} \Delta_t(f) \mulebdii{t}}_{L^\infty(\brd)} <\infty  
\end{equation}
defines an equivalent norm on $\bessel r {L^\infty(\brd)}$.
\end{thm}
Combining Lemma~\ref{lem:weak-bessel-norm} with Theorem~\ref{thm-hypersing-1d} we obtain the first statement of the following theorem.
\begin{thm}
  \label{thm-hypersingint}
Let \mA\ be a \BA\ and $\Psi$ a $C_w$-automorphism group acting on it. For $0 < r <2 $ the norm $\norm{a}_{\bessel r \mA} $ is equivalent to
\begin{equation}\label{eq:hypsingnorm}
 \norm{a}_\mA + \sup_{\epsilon > 0}\bignorm{ \int_{\abs t_2 \geq \epsilon} \abs{t}_2^{-r} \Delta_t(a) \mulebdii{t}}_\mA.
 \end{equation}
This norm is further equivalent to
\[
\norm{a}_\mA + \sup_{\epsilon > 0}\bignorm{ \int_{\epsilon \leq \abs t _2 \leq 1} \frac{\Delta_t(a)} { \abs{t}_2^{r}}\mulebdii{t}}_\mA.
\]
\end{thm}
 \begin{proof}
   We only show the second statement. As
\begin{align*}
  \bignorm{ \int_{\epsilon \leq \abs t _2 } \frac{\Delta_t(a)} { \abs{t}_2^{r}}\mulebdii{t}}_\mA &\leq
  \bignorm{ \int_{\epsilon \leq \abs t _2 \leq 1} \frac{\Delta_t(a)} { \abs{t}_2^{r}}\mulebdii{t}}_\mA +
  \bignorm{ \int_{\abs t _2 \geq  1} \frac{\Delta_t(a)} { \abs{t}_2^{r}}\mulebdii{t}}_\mA \\
  &\leq
  \bignorm{ \int_{\epsilon \leq \abs t _2 \leq 1} \frac{\Delta_t(a)} { \abs{t}_2^{r}}\mulebdii{t}}_\mA +
  (1+M_\Psi)\norm{a}_\mA \int_{\abs t _2 \geq  1}  \abs{t}_2^{-r}\mulebdii{t} \\
& \leq C( \norm{a}_\mA +\bignorm{ \int_{\epsilon \leq \abs t _2 \leq 1} \frac{\Delta_t(a)} { \abs{t}_2^{r}}\mulebdii{t}}_\mA ) \, ,
\end{align*}
the proof of the other inequality works in a similar way.
 \end{proof}
Next we  compare  Bessel potential spaces with Besov spaces. 
\begin{prop} \label{prop:besselvsbesov} If \mA\ is \BA\ with $C_w$-automorphism group $\Psi$, then
  \[
\besov 1 r \mA \inject \bessel r \mA \inject \besov \infty r \mA \,\quad \text{ if } r>0. 
\]
\end{prop}

\begin{proof}
 For the proof of the embedding $\bessel r \mA \inject \besov \infty r \mA $ let $a \in \bessel r \mA$ with $a=\mG_r*y$, $y \in \mA$. The seminorm  $\abs{a}_{\besov \infty r \mA}$ can be estimated for $k > \floor{r}$ as
    \begin{align*}
      \abs{a}_{\besov \infty r \mA}
       &= \sup_{\abs t \neq 0} \frac{\norm{\Delta^k_t(\mG_r *y)}_\mA}{\abs t ^r}
        = \sup_{\abs t \neq 0} \norm {\frac{\Delta^k_t(\mG_r)}{\abs t ^r} *y}_\mA\\
        &\leq \sup_{\abs t \neq 0}\norm{\frac{\Delta^k_t(\mG_r)}{\abs t ^r}}_{L^1(\brd)} \norm{y}_\mA 
          = \norm{\mG_r}_{\besov \infty r {L^1}} \norm{a}_{\bessel r \mA } \,,
    \end{align*}
and this is the desired embedding.
  We still have to verify  the first inclusion. Assume first that $0 < r <1$.  By Theorem~\ref{thm-hypersingint}, for an $a \in \bessel r \mA$
  \begin{align*}
    \norm{a}_{\bessel r \mA}
    &\asymp \norm{a}_\mA + \sup_{\epsilon > 0}\bignorm{ \int_{\abs t_2 \geq \epsilon} \abs{t}_2^{-r} \Delta_t(a) \mulebdii{t}}_\mA\\
    &\leq \norm{a}_\mA +  \int_{\brd} \abs{t}_2^{-r} \norm{\Delta_t(a)}_\mA \mulebdii{t}\\
   & =\norm{a}_{\besov 1 r \mA} \,.
  \end{align*}
In the general case we proceed by induction. Assume that the statement is true for all positive values up to $s>0$, and $s < r < s+1$. Then
\[
\besov 1 r \mA = \besov 1 {r-s} {\besov 1 s \mA} 
\subseteq \bessel {r-s}{\besov 1 s \mA}
\subseteq \bessel {r-s}{\bessel s \mA}
=\bessel r \mA \,,
\]
where we have used the reiteration theorems for the Bessel and the Besov spaces (Theorem~\ref{thm-reiteration}). 
\end{proof}
% \begin{rem}
%    In order for this proof to be useful it is necessary to know that we ``play fair'' here: For the proof of Theorem~\ref{thm-hypersing-1d}  only the embedding of Lemma~\ref{lem:besselinbesov} is needed.
% \end{rem}
%
Another application of the reiteration theorem and the representation of the norm of $\bessel r \mA$ by the hypersingular integral \eqref{eq:hypsingnorm} shows how Besov spaces and Bessel potential spaces interact.
\begin{prop} If \mA\ is a \BA\ with $C_w$-automorphism group $\Psi$, then for all  $r, s >0$ and $1 \leq p \leq\infty$
  \label{prop:besselonbesov}
  \begin{equation}
    \label{eq:besselonbesov}
    \bessel r {\besov p s \mA} =   \besov p s {\bessel r \mA} = \besov p {r+s} \mA \,.
  \end{equation}
\end{prop}
\begin{proof}
Again, we need to know first that $\Psi$ is a $C_w$-automorphism group on $\besov p r \mA$. If $p <\infty$ then $C(\besov p r \mA)=\besov p r \mA$ by Proposition~\ref{prop:besovcontpart}. If $p=\infty$ the assertion follows from
\[
\norm{a}_{\besov \infty r \mA}= \sup_{\abs t \neq 0}\sup \set {\inprod{a', \frac{\Delta^k_t(a)}{\abs t ^r }} \colon a' \in \mA'_\Psi, \norm {a'}_{\mA'} \leq 1} \,.
\]
The details are similar to the proof of the analogue statement in Proposition~\ref{reitbessel} and are left to the reader.

Using the reiteration theorems for Bessel potential spaces and Besov spaces, it suffices to prove the proposition only for $0<r,s <1$.
  We show first that $\bessel r {\besov p s \mA} \inject \besov p s {\bessel r \mA}$. Assume that $a \in \bessel r {\besov p s\mA}$, so $a=\mG_r*y$ with $y \in \besov p s \mA$.  We obtain the following estimate.
   \begin{align*}
     \norm{a }_{\besov p s {\bessel r \mA}}^p 
     =& \int_{\brd} \frac{\norm{\Delta_t (a)}_{\bessel r \mA} ^p}{\abs t ^ {s p}} \mulebd{t} \\
     =& \int_{\brd} \frac{\norm{ \Delta_t (\mG_r *y)}_{\bessel r \mA} ^p}{\abs t ^ {s p}} \mulebd{t} \\
     =& \int_{\brd} \frac{\norm{\mG_r* \Delta_t ( y)}_{\bessel r \mA} ^p}{\abs t ^ {s p}} \mulebd{t} \\
     =& \int_{\brd} \frac{\norm{\Delta_t ( y)}_{ \mA} ^p}{\abs t ^ {s p}} \mulebd{t} \\
  %\leq & \int_{\brd} \frac{\norm{\mG_r}_{L^1(\brd)}^p \norm{ \Delta_t (y)}_\mA^p}{\abs t ^ {s p}} \mulebdii{t} \\
     = & \norm{y}_{\besov p s {\mA}}^p=\norm{\mG_r * y}_{\bessel r {\besov p s {\mA}}} ^p  \,.
   \end{align*}
Now let $a=\mG_r*y \in \bessel r {\besov p s \mA} $. Then
\begin{align*}
\norm{a}_{ \bessel r {\besov p s \mA}}^p = & \norm{\mG_r*y}_{ \bessel r {\besov p s \mA}}^p
=\norm{y}_ {\besov p s \mA}^p 
 = \int_{\brd} \frac{\norm{\Delta_t ( y)}_\mA ^p}{\abs t ^ {s p}} \mulebdii{t} \\
 =& \int_{\brd} \frac{\norm{\Delta_t (\mG_r * y)}_{\bessel r \mA} ^p}{\abs t ^ {s p}} \mulebdii{t} \\
 =&  \norm{\mG_r*y }_{\besov p s {\bessel r \mA}} =  \norm{a }_{\besov p s {\bessel r \mA}}\,.
\end{align*}
 
Consequently $\bessel r {\besov p s \mA} =   \besov p s {\bessel r \mA} $.
Finally, Proposition~\ref{prop:besselvsbesov} implies that
\[
\besov p s {\besov 1 r \mA} \inject \besov p s {\bessel r \mA} \inject \besov p s {\besov \infty r \mA}  \,,
\]
and the first and last space in this chain equal $\besov p {r + s} \mA$ by the reiteration theorem for Besov spaces (Theorem~\ref{thm-reiteration}).
\end{proof}

\subsubsection*{Algebra Properties}
\label{sec:algebra-properties}
 The characterization of Bessel potential spaces by a hypersingular integral  yields the \BA\ properties of  $\bessel r \mA$. %For the proof we need one more norm equivalence.
% \begin{lem}\label{lem-trunc-hypersing} If $r>0$, then
\begin{thm}\label{prop:thm-besspot_baic}
   If \mA\ is a \BA\ with $C_w$-group $\Psi$, then
 the Bessel potential space $\bessel r  \mA$ is a Banach subalgebra of \mA\ for every $r>0$.
   Moreover,   $\bessel r \mA$ is \IC\ in \mA.
     \end{thm}
For functions in $\bessel r {L^\infty(\brd}$ this result is in Strichartz~\cite{strichartz67}. 
  \begin{proof} 
    We treat the case $r < 1 $ first. Let $a,b \in \bessel r \mA$. Using 
\[
\Delta_t(ab) = \Delta_t(a)\Delta_t(b) + a \Delta_t(b) + \Delta_t(a) b
\]
we obtain 
\begin{equation}\label{eq:besselalgestimate}
\begin{split}
\bignorm{ \int_{\epsilon \leq \abs t _2 \leq 1}  \frac{\Delta_t(ab)}{\abs{t}_2^{r}}\mulebdii{t}}_\mA &\leq 
      \bignorm{\int_{\epsilon \leq \abs t _2 \leq 1}  \frac{\Delta_t(a)\Delta_t(b)}{\abs{t}_2^{r}}\mulebdii{t}}_\mA  \\
+\bignorm{a \int_{\epsilon \leq \abs t _2 \leq 1} &\frac{ \Delta_t(b)}{\abs{t}_2^{r}}\mulebdii{t}}_\mA +\bignorm{\Bigl(\int_{\epsilon \leq \abs t _2 \leq 1} \frac{\Delta_t(a)}{\abs{t}_2^{r}}\mulebdii{t} \Bigr)b}_\mA .
\end{split}
\end{equation}
The second and third term of the expression on the right hand side of the inequality are dominated  by
\[
\norm{a}_\mA \norm{b}_{\bessel r \mA}+ \norm{a}_{\bessel r \mA}\norm{b}_\mA \lesssim \norm{a}_{\bessel r \mA}\norm{b}_{\bessel r \mA} \,.
\]

For the estimation of the first term in~\eqref{eq:besselalgestimate} we use the embedding $\bessel r \mA \inject \besov \infty r \mA$  (Proposition~\ref{prop:besselvsbesov}), so
$
\norm{\Delta_t a}_\mA \lesssim \abs t _2 ^r \norm{a}_{\bessel r \mA},
$
with a similar estimate for $b$.
Therefore
\[
 \bignorm{\int_{\epsilon \leq \abs t _2 \leq 1}  \frac{\Delta_t(a)\Delta_t(b)}{\abs{t}_2^{r}}\mulebdii{t}}_\mA \lesssim
\norm{a}_{\bessel r \mA}  \norm{b}_{\bessel r \mA}\int_{0 \leq \abs t _2 \leq 1} \abs{t}_2^r \mulebdii{t}  \leq C_r \norm{a}_{\bessel r \mA}  \norm{b}_{\bessel r \mA} \,,
\]
and $C_r$ does not  depend on $\epsilon$.
Combining the estimates  we have proved that 
\[
\norm{ab}_{\bessel r \mA} \lesssim \norm{a}_{\bessel r \mA}\norm{b}_{\bessel r \mA}.
\]

For the verification of the \IC ness of $\bessel r \mA$ in $\mA$ we use a similar argument: Expand the identity (\eqref{eq:invderiv}) to obtain
\begin{equation*}
  \Delta_t(a^{-1})=-\Delta_t(\inv{a})\Delta_t(a) \inv{a} - \inv a \Delta_t(a) \inv a.
\end{equation*}
 So
\begin{multline}\label{eq:besselic}
  \bignorm{\int_{\epsilon \leq \abs t _2 \leq 1}  \frac{\Delta_t(\inv a)}{\abs{t}_2^{r}}\mulebdii{t}}_\mA
  \leq \bignorm{\int_{\epsilon \leq \abs t _2 \leq 1}  \frac{\Delta_t(\inv a) \Delta_t(a) \inv a}{\abs{t}_2^{r}}\mulebdii{t}}_\mA \\
    +   \bignorm{\int_{\epsilon \leq \abs t _2 \leq 1}  \frac{\inv a \Delta_t( a) \inv a}{\abs{t}_2^{r}}\mulebdii{t}}_\mA   \,.
\end{multline}
As $a \in \besov \infty r \mA$, we know that
$
\norm{\Delta_t( a)}_\mA \lesssim \abs{t}_2^r  \norm{ a}_{\besov \infty r \mA}$ , and, as $\besov \infty r \mA$ is \IC\ in \mA,
$\norm{\Delta_t(\inv a)}_\mA \lesssim \abs{t}_2^r  \norm{\inv a}_{\besov \infty r \mA}$.%\lesssim \abs t^r \norm{\inv a}^2_\mA \norm{a}_{\besov \infty r \mA},

%the last inequality follows by taking norms in (\eqref{eq:invderiv}).
The first term on the right hand side of \eqref{eq:besselic} can be dominated by

\[
\int_{\epsilon \leq \abs t _2 \leq 1} \frac{\norm{\Delta_t(\inv a)}_\mA \norm{\Delta_t a}_\mA \norm{\inv a}_\mA}{\abs t _2 ^r} \mulebdii{t} \lesssim
\norm{\inv a}_{\besov \infty r \mA} \norm{a}_{\besov \infty r \mA}\norm{a}_{ \mA}
%\lesssim
%\norm{\inv a}^3_\mA \norm{a}^2_{\bessel r \mA} \,.
\]
The second term can be estimated as

\begin{align*}
  \bignorm{\int_{\epsilon \leq \abs t _2 \leq 1}\frac{\inv a \Delta_t(a) \inv a}{\abs t _2 ^r} \mulebdii{t}}_\mA 
 =& \bignorm{\inv a \Bigl(\int_{\epsilon \leq \abs t _2 \leq 1}\frac{ \Delta_t(a) }{\abs t _2 ^r} \mulebdii{t}\Bigr) \inv a}_\mA \\
\lesssim & \norm{\inv a}^2_\mA \norm{a}_{\bessel r \mA}.
\end{align*}
As $\besov \infty r \mA$ is \IC\ in \mA\ we obtain the \IC ness of $\bessel r \mA$ in \mA.

If $r \geq 1$ we can proceed by induction. Assume that we have already proved that $\bessel s \mA$ is \IC\ in \mA, and $s < r < s+1$. By what we have just proved 
$\bessel r \mA = \bessel {r-s} {\bessel s \mA}$ is \IC\ in ${\bessel s \mA}$. As ${\bessel s \mA}$ is \IC\ in \mA\ by hypotheses we are done.
\end{proof}
%   \begin{rems}
%     \begin{enumerate}\hspace{1cm}
% \item  We have proved more than stated. As for Besov spaces we obtain
%       \emph{norm controlled inversion}, that is,
%       estimates for $\norm{\inv a}_{\bessel r \mA}$ in terms of
%       $\norm{\inv a}_{\mA}$. See Nikolski~\cite{Nikolski99,nikolski01}
%       for a discussion.
%       \item
%         It would be interesting to obtain a similar theory for more general forms of the kernel $\mG_r$ (see also the next paragraph).
%     \end{enumerate}
%       \end{rems}

      \subsubsection{Application to Weighted Matrix Algebras}
      \label{sec:appl-weight-matr}
We call  $ v^*_r(k) = (1+\abs{ 2 \pi k}_2^2)^{r/2}$ for $r>0$ the \emph{Bessel weight} of order $r$.  If \mA\ is a \BS\ of matrices, we say that a matrix $A$ is in the weighted matrix space $\mA_{v_r}$, where $v_r$ is the standard polynomial weight $v_r(k)=(1+\abs k)^r$,  if the matrix with entries $A(k,l)v_r(k-l)$ is in \mA. The norm in 
$\mA_{v_r}$ is $\norm{A}_{\mA_{v_r}}=\norm{(A(k,l)v_r(k-l))_{k,l \in \bzd}}_\mA $. In a similar way we introduce $\mA_{v^*_r}.$
\begin{prop}
  \label{prop:besselmatrixspace}
  If $\mA$ is a \hmg\ \MA, and $\chi$ is a $C_w$- group on \mA, then 
\[
\mA_{v^*_r}=\bessel r \mA \,.
\]
\end{prop}
\begin{proof}
By definition $A$ is in $\bessel  r \mA$, if there is a $A_0 \in \mA$ such that $A= \mG_r *A_0$. This is equivalent to
\[
\hat A (k) = (1+\abs{ 2 \pi k}^2)^{-r/2} \hat A_0(k) \,,
\]
or $\hat A_0(k) = (1+\abs{ 2 \pi k}^2)^{r/2} \hat A(k)$,
and therefore
\[
\norm{A}_{\bessel r \mA}=\norm{A_0}_\mA = \norm{A}_{\mA_{v^*_r}} \,,
\]
 i.e., $A \in \mA_{v^*_r}$.
\end{proof}
%If $\mA$ is a \HMA, and $\chi$ is a $C_w$ group on \mA, the weighted \MS s $\mA_{v^*_r}$ are \IC\ \MA s.
\begin{prop}
  \label{prop:weightes_MAs}
If \mA\ is a \HMA, $\chi$ is a $C_w$- group on \mA, and $v^*_r$, $r>0$, is a Bessel weight, then $\mA_{v^*_r} = \bessel r \mA$ is a \MA. This algebra is \IC\ in \mA.%\, and 
%$\norm{\inv A}_{\mA_{v^*_r}}$ is controlled by $\norm{\inv A}_\mA$.
\end{prop}
\begin{proof}
  This is an application of Theorem~\ref{prop:thm-besspot_baic}.
\end{proof}
Proposition \ref{prop:weightes_MAs} applies in particular to the weighted subalgebras of \bop.

For solid \MA s the standard polynomial weights $v_r$ can be taken instead of $v^*_r$.
\begin{cor}
  If \mA\ is a solid \MA, and $\chi$ is a $C_w$- group on \mA,  then $\mA_{v_r}$ is an \IC\ subalgebra of \mA.
\end{cor}
%This result should be compared with Proposition~\ref{prop:solid-ma} and Proposition~\ref{prop-solidmazd}.

We state the results of Proposition \ref{prop:besselvsbesov} and Proposition \ref{prop:besselonbesov} for  weighted \MA s.
\begin{prop}
  \label{prop:besselbesovMA}
  If $\mA$ is a \HMA, and $r,s >0$, then
  \begin{equation*}
  \begin{split}
    &\besov 1 r \mA \inject \mA_{v^*_r} \inject \besov r \infty \mA, \\
    &\besov p r {\mA_{v^*_s}} = \app p r {\mA_{v^*_s}} = (\besov p r \mA)_{v^*_s} = \besov p {r+s} \mA =  \app p {r+s} \mA.
  \end{split}
\end{equation*}
\end{prop}
\begin{ex}
  For the Schur algebras $\schur [p] r$ we obtain 
  \[
  \app q s {\schur [p] r}=\app q {s+r}{\schur [p] 0} \, .
  \]
\end{ex}

% **nach unten verschieben?***In this case we can argue as follows: $\app q s {\schur [p] r} =\app q s {\bessel r {\schur [p] 0}} =\app q {s+r}{\schur [p] 0}$ by Proposition~\ref{prop:besselbesovMA}. A direct argument using the \LP-like representation of Proposition~\ref{prop:lpnorm_solid_ms} is also possible. 

% begin{rem}
%   We have not identified the Besov spaces related to a \MA\ yet. Though this is not difficult, we postpone it to Section~\ref{sec:hz-spaces-ma}. However, from the inclusion relations in Prop~\ref{prop-inclusion} and the first relation in Proposition~\ref{prop:besselbesovMA} we conclude that for $\epsilon >0$ and $r>0$
% \begin{equation*}
%       \mA_{v^*_{r+\epsilon}} \inject \besov 1 r \mA \inject \mA_{v^*_r} % \\
%      \inject \besov \infty r \mA \inject \mA_{v^*_{r-\epsilon}} \,,
%   \end{equation*}
% so Besov spaces are related to \odd. The last embedding in the relation above follows directly from the characterization of the norm of $\bessel {r-\epsilon} \mA$ by the hypersingular integral \eqref{eq:hypsingnorm}.
% \end{rem}
% ****
\appendix
%\section{}
\section{Proof of the Reiteration theorem}
\label{sec:proof-reit-theor}
We give a proof of Theorem~\ref{thm-reiteration}. We need some properties of the moduli of smoothness.
\begin{lem}\label{lem-modsmooth}
  If $l,k \in \bn$, $l\geq k$, $t \in \brd$ and $h>0$, then
  \begin{enumerate}
  \item \label{smoothordera}
    $\norm{\Delta^l_t (x)}_\mX \leq (M_\Psi+1)^k \norm{\Delta^{l-k}_t
      (x)}_\mX \text{ and } \omega^l_h (x) \leq (M_\Psi+1)^k
    \omega^{l-k}_h (x) \,,$
  % \item \label{stepsizea}
  %   \[\norm{\Delta^k_{2t}(x)}_\mX \leq
  %   (M_\Psi+1)^k\norm{\Delta^k_{t}(x)}_\mX \text{ and } \omega^k_{2h}
  %   (x) \leq (M_\Psi+1)^k \omega^k_h (x) \,,\]
  % \item \label{contstepa} If $\lambda >0$ then
  %   \[ \norm{\omega^k_{\lambda t}(x)}_\mX \leq (M_\Psi+1)^k
  %   (\lambda+1)^k\norm{\Delta^k_{t}(x)}_\mX \,. \]
  % \item \label{marchauda} (Marchaud inequality)
  %   \[
  %   \omega^k_h (x) \leq C h^k \int_h^\infty \frac{\omega^l_u (x)}{u^k}
  %   \frac{du }{u} .
  %   \]
  % \item \label{intmodconta} The averaged modulus of smoothness
  %   \[
  %   \overline{w}^k_h(x)= h^{-d} \int_{\abs t \leq h} \norm{\Delta^k_t
  %     x}_\mX \, dt
  %   \]
  %   is equivalent to the ``standard'' modulus of smoothness:
  %   $\overline{w}^k_h(x) \asymp \omega^k_h(x)$ \cite[Lemma
  %  6.5.1]{DeVore93}.
  \item \label{compmodconta} 
    \[
    \omega^k_t (x) \asymp \sup_{\substack{\abs{h_j}\leq t \\ 1 \leq j
        \leq k}} \norm{\bigl(\prod_{j=1}^k \Delta_{h_j}\bigr) x}_\mX.
    \]
  % \item \label{modderiva} If $a \in C^k(\mX)$, then
  %   \[
  %   \omega_h^{k+l}(x) \leq C \sup_{\abs \alpha =
  %     k}\omega_h^l(\delta^\alpha(x))
  %   \]
  \end{enumerate}
\end{lem}
  The proof of (\ref{smoothordera}) is an %, (\ref{stepsizea}) and (\ref{contstepa}) are
  easy calculation in complete analogy to the corresponding properties of the moduli of smoothness for functions. See, e.g.,~\cite{DeVore93}. Item~\ref{compmodconta} is proved in~\cite[5.4.11]{bennett88}.

\begin{proof}[Proof of the Reiteration theorem]
  We assume first that $x$ is in $\Lambda^q_{s+r}(\mA)$ and estimate $\norm{x}_{ \Lambda^q_s(\Lambda^p_r(\mA))}$.  As $
  \norm{x}_{\Lambda^q_s(\Lambda^p_r(\mA))}= \norm{x}_{\Lambda^p_r(\mA)}+ \abs{x}_{\Lambda^q_s(\Lambda^p_r(\mA))}$ and  the inclusion relations of Besov spaces imply that $\norm{x}_{\besov p r \mA} \leq C \norm{x}_{\besov q {r+s} \mA}$,  it  suffices to estimate
  $\abs{x}_{\Lambda^q_s(\Lambda^p_r(\mA))}$.

  Assume that $\floor r < m$ and $\floor s <n$, $m,n \in \bn$. Using the norm equivalences in \eqref{eq:besovnormeqs} we can write
  \begin{equation}
    \label{eq:reit1}
    \begin{split}
    \abs{x}_{\besov q s {\besov p r \mA}} &\asymp 
    \Biggl\lbrace \int_{\br^+} \Bigl[ h^{-s} \omega^{n+m}_h(x, \besov p r \mA) \Bigr ] ^q \muleb{h}\Biggr \rbrace^{1/q}\\
   & =
    \norm{h^{-s} \omega^{n+m}_h (x, \besov p r \mA)}_{L^q_*} \,,
  \end{split}
\end{equation}
where $\norm{f}_{L^q_*}= \bigl(\int_0^\infty f(t)^q \muleb {t}\bigr)^{1/q}$.
  An estimate of the modulus of smoothness is
  \begin{equation}
    \label{eq:reit2}
    \begin{split}
      \omega^{n+m}_h(x, \besov p r \mA)& = \sup_{\abs u \leq h} \norm{\Delta^{n+m}_u x}_{\besov p r \mA}\\
%       &\asymp \sup_{\abs u \leq s} \Biggl \lbrace \norm{\Delta^{b+a}_u x}_\mA + \biggl[\int_{\br^+} t^{-r p}\sup_{\abs v \leq t} \norm{\Delta^{b+a}_u
%         \Delta^{b+a}_v x}_\mA ^p \muleb{t} \biggr]^{1/p}
%       \Biggr \rbrace \\
%       &\lesssim \sup_{\abs u \leq s} \norm{\Delta^{b+a}_u x}_\mA + \sup_{\abs u \leq s}
%       \biggl[\int_{\br^+} t^{-r p}\sup_{\abs v \leq t} \norm{\Delta^{b+a}_u \Delta^{b+a}_v x}_\mA ^p \muleb{t} \biggr]^{1/p}\\
      &\leq \sup_{\abs u \leq h} \norm{\Delta^{n+m}_u x}_\mA +
      \sup_{\abs u \leq h}  \abs{\Delta^{n+m}_u x}_{\besov p r \mA} \\
      & \lesssim \sup_{\abs u \leq h} \norm{\Delta^{n+m}_u x}_\mA + \sup_{\abs u \leq h} \abs{\Delta^{n+m}_u x}_{\besov 1 r \mA} \,,
    \end{split}
  \end{equation}
 where the last inequality uses the embedding $\besov 1 r \mA \inject \besov p r \mA$ for $p \geq 1$.

  Inserting this estimate into~\eqref{eq:reit1} we obtain
  \begin{equation}
    \label{eq:reit3}
    \abs{x}_{\besov q s {\besov p r \mA}} \lesssim \abs{x}_{\besov q s \mA} +
    \norm{h^{-s} \sup_{\abs u \leq h} \abs{\Delta^{n+m}_u x}_{\besov 1 r \mA}}_{L^q_*}  \,.
  \end{equation}
  With
  $
  \phi(v,u)= \norm{\Delta^{n+m}_v \Delta^{n+m}_u x}_\mA
  $
  the $L^q_*$-norm in \eqref{eq:reit3} can be rewritten as
  \begin{equation*}
    \label{eq:reit4}
    \begin{split}
      % \norm{s^{-s} \sup_{\abs u \leq s} \abs{\Delta^{b+a}_u x}_{\besov 1 r \mA}}_{L^q_*} &= &\biggl[ \int_{\br^+} s^{-s q} \sup_{\abs u \leq
      %   s}
      % \biggl( \int_{\br^+} t^{-r}\sup_{\abs v \leq t} \phi(u,v) \muleb{t} \biggr)^q  \muleb{s} \biggr]^{1/q} \\
      & \norm{h^{-s}\sup_{\abs u \leq h}
        \int_{\br^+} t^{-r}\sup_{\abs v \leq t} \phi(v,u) \muleb{t}}_{L^q_*}\\
      &\leq \norm{h^{-s}\sup_{\abs u \leq h} \int_0^h t^{-r}\sup_{\abs v \leq t} \phi(v,u) \muleb{t}}_{L^q_*} + \norm{h^{-s}\sup_{\abs u \leq h}
        \int_h^\infty  t^{-r}\sup_{\abs v \leq t} \phi(v,u) \muleb{t}}_{L^q_*}\\
      % \Biggl[ \int_{\br^+} s^{-s q} \sup_{\abs u \leq s} \biggl(\int_s^\infty t^{-r}\sup_{\abs v \leq t} \phi(v,u) \muleb{t} \biggr)^q \muleb{s}
      % \Biggr]^{1/q} \\
      = :\text{I} +\text{II}.
    \end{split}
  \end{equation*}
  We can  estimate the first term further using Hardy's inequality.
  \begin{equation*}
    \label{eq:reit5}
    \begin{split}
      \text{I}^q &= \int_0^\infty h^{-s q} \sup_{\abs u \leq h} \biggl[ \int_0^h \sup_{\abs v \leq t} t^{-r}\phi(v,u) \muleb{t} \biggr] ^q
      \muleb{h} \\
      &\leq \int_0^\infty h^{-s q} \biggl[ \int_0^h \sup_{\abs {v},\abs{u} \leq h} t^{-r}\phi(v,u) \muleb{t} \biggr] ^q \muleb{h}  \\
      &\overset {(*)}{\ls } \int_0^\infty \biggl( t^{-(r+s)} \sup_{\abs {v},\abs{u} \leq t}\phi(v,u)   \biggr) ^q \muleb{t}\\
      & \overset{(**)}{\ls} \int_0^\infty \biggl( t^{-(r+s)}\omega^{2(n+m)}_t(x,\mA)\biggr)^q \muleb{t} = \abs{x}^q_{\besov q {r + s} \mA} \,,
    \end{split}
  \end{equation*}
  $(*)$ by Hardy's inequality, and $(**)$ using Lemma~\ref{lem-modsmooth}(\ref{compmodconta}).  For the second term we use (\ref{smoothordera}) of Lemma~\ref{lem-modsmooth} to get
  \[
  \phi(v,u)=\norm{\Delta^{n+m}_v\Delta^{n+m}_u x}_\mA \lesssim \norm{\Delta^{n+m}_u x}_\mA.
  \]
  Then $\sup_{\abs v \leq t} \phi(u,v)$ is independent of $t$, and 
  \begin{equation}
    \label{eq:reit4}
    \begin{split}
      \text{II}^q &\lesssim \int_0^\infty \biggl( h^{-s} \sup_{\abs u \leq h} \int_h^\infty t^{-r}\norm{\Delta^{n+m}_u x}_\mA \muleb{t} \biggr)^q
      \muleb{h}\\
      &= \int_0^\infty h^{-(r + s) q}\sup_{\abs u \leq h} \norm{\Delta^{n+m}_u x}^q_\mA \muleb{h} =\abs{x}^q_{\besov q {r + s} \mA}.
    \end{split}
  \end{equation}
  I and II together give the desired estimate.

  For the converse assume that $x \in \besov q s {\besov p r \mA}$. Then, 
  \begin{equation}\label{eq:22}
  \begin{split}
    \abs{x}^q_{\besov q {r + s } \mA} &\asymp \int_{\brd} \biggl(\abs{t}^{-(r+s)} \norm{\Delta_t^{n+m}x}_\mA\biggr)^q \mulebd{t}\\
   & \asymp \int_{\brd} \abs{t}^{-r q }\biggl( \int_{\abs \eta \geq \abs t}\abs{\eta}^{-s p}\norm{\Delta_t^{n+m}x}^p_\mA \mulebd{\eta} \biggr)^{q/p} \mulebdii{t} \,,
  \end{split}
 \end{equation}
 where we have used  
$
\abs t ^{-s} \asymp \bigl(\int_{\abs \eta \geq \abs t }\abs{\eta}^{-s p} \mulebdii{\eta} \bigr)^{1/p}
$ 
for the last equivalence.
As 
$
\norm{\Delta^{n+m}_tx}_\mA \leq \sup_{\abs v \leq \abs {\eta}}\norm{\Delta^m_v\Delta^n_t}_\mA
$
for $\abs \eta \geq \abs t$ , we can dominate the right hand side of (\ref{eq:22}) by
\[
\begin{split}
&\int_{\brd} \abs{t}^{-r q }\sup_{\abs u \leq \abs t}\biggl( \int_{\abs \eta \geq \abs t}\abs{\eta}^{-s p}\sup_{\abs v \leq \abs \eta}\norm{\Delta_v^n\Delta_u^{m} x}^p_\mA \mulebd{\eta} \biggr)^{q/p} \mulebd{t}\\
&\leq
\int_{\brd} \abs{t}^{-r q }\sup_{\abs u \leq \abs t}\biggl( \int_{\brd}\abs{\eta}^{-s p}\sup_{\abs v \leq \abs \eta}\norm{\Delta_v^n\Delta_u^{m} x}^p_\mA \mulebd{\eta} \biggr)^{q/p} \mulebd{t}\\
&\leq
\int_{\brd} \abs{t}^{-r q }\sup_{\abs u \leq \abs t}(\abs{\Delta_u^m x}_{\besov p s \mA})^q \mulebd{t}\\
&\leq \abs{x}^q_{\besov q r {\besov p s \mA}}.
\end{split}
\]
\end{proof}

\section{Jackson Bernstein Theorem}
\label{sec:jacks-bernst-theor}

%  With the existence of approximating kernels we can now state a Jackson-type theorem for automorphism groups.
%
  \begin{prop}[{\cite[5.12]{grkl10}}] \label{prop:jackson1} Let $a \in \mA$ and $\sigma>0$.
\begin{enumerate}
   \item \label{jackhoeld} There is a $\sigma$-\BL\ element $a_\sigma \in C(\mA )$ such that
    \begin{equation*}
      \norm{a-a_\sigma}_\mA \leq C \omega_{1/\sigma}(a)
    \end{equation*}
    with $C$ independent of $\sigma$ and $a$.
   \item \label{jackdiff} If $\delta^\alpha (a) \in C(\mA)$, for all multi-indices $\alpha$ with $\abs \alpha =k$ then there exists a $\sigma$-\BL\ element $a_\sigma \in \mA$ such that
    \begin{equation*}
      \norm{a-a_\sigma}_\mA \leq C \sigma^{-k}\sum_{\abs\alpha =k}\omega^2_{1/\sigma}(\delta^\alpha a)\, .
    \end{equation*}
  \end{enumerate}
\end{prop}
  \begin{cor} \label{p:ch47} If $a \in \besov p r \mA$ for $r >0$, then $a \in \aps  r (\mA)$.
  \end{cor}
\begin{proof} We use the integral version of the norm for an approximation space in~\eqref{eq:appspace} and assume that $1 \leq p <\infty$. The proof for $p=\infty$ is simpler and done in \cite{grkl10}.
  
  Assume first that  $0<r < 1$. Then, by Proposition~\ref{prop:jackson1}(\ref{jackhoeld}),
\[
\int_1^\infty \bigl(E_\sigma(a)\sigma^r \bigr)^p\muleb{\sigma} \leq C \int_0^1 \bigl(\omega_\tau(a)\tau^{-r}\bigr)^p \muleb{\tau}
\leq C \abs{a}^p_{\besov p r \mA}  \,,
\]
and so the approximation norm is dominated by the Besov norm.

Likewise, if $r =k+ \eta, 0 < \eta \leq 1$, and $k \in \bn$,  Proposition~\ref{prop:jackson1}(2) yields
 \begin{align*}
   \int_1^\infty \bigl(E_\sigma(a)\sigma^r \bigr)^p\muleb{\sigma} 
\leq   C \sum_{\abs \alpha = k} \int_0^1 \bigl(\omega_\tau^2(\delta^\alpha (a))\tau^{-\eta}\bigr)^p \muleb{\tau} \,
 \end{align*}
and again  $\norm{a}_{\aps  r (\mA)}$ is dominated by the Besov norm.% (see Proposition \ref{prop-equivalent-norms}).
\end{proof}
Before proving the converse implication in Theorem~\ref{prop:jacksonbernstein}, i.e., the Bernstein-type result,
we need a mean-value property of automorphism groups.
\begin{lem}[{\cite[5.15]{grkl10}}] \label{lem:ch45} 
 If $a$ is $\sigma$-\BL, then
  \begin{equation}
    \label{eq:49}
    \norm{\Delta_t a }_\mA \leq C\sigma\,  \abs{t}\, \norm{a}_\mA \, .
  \end{equation}
\end{lem}
 \begin{prop}\label{p:ch46} %[ernstein] 
  Let $a \in \mA$, and $r >0$, $1 \leq p \leq \infty$. If $a \in \app p r \mA$, then $a \in
  \besov p r \mA$.
\end{prop}
\begin{proof}
  We adapt a standard proof~\cite{Butzer71} and verify the statement for $p < \infty$.
%  We will work with the discrete Besov norm (see Proposition~\ref{prop-equivalent-norms})
% \[
% \norm{a}_{\besov p r \mA} 
% \asymp \norm{a}_\mA +  \biggl(\sum_{l=0}^\infty \bigl(2^{l r} \omega^{m}_{2^{-l}}(a) \bigr)^p \biggr)^{1/p}   \,,
% \]
% where $m > \floor r$.
If $a \in \app p r \mA$, the representation theorem of approximation theory (see, e.g~\cite[3.1]{Pietsch81}) implies that
\begin{equation}
  \label{eq:blrep}
  a=\sum_{k=0}^\infty a_k,\quad \text{with  } a_k \in X_{2^k} \quad \text{and  } \sum_{k=0}^\infty 2^{kr p}\norm{a_k}_\mA^p  < \infty \,,
\end{equation}
where $(X_{\sigma})_{\sigma \geq 0}$ is the approximation scheme of \BL\ elements, and 
\[
\norm{a}_{\app p r \mA} \asymp \bigl(\sum_{k=0}^\infty 2^{kr p}\norm{a_k}_\mA^p \bigr)^{1/p} \,,
\]
where the infimum is taken over all admissible representations as in \eqref{eq:blrep}.
An  application of H\"olders inequality shows that  $\sum_{k=0}^\infty a_k$ is convergent in \mA.
Note that \eqref{eq:blrep} implies that 
$%\begin{equation}
  \label{eq:blcoeff}
  \norm{a_k}_\mA \leq C 2^{-k r}
$%\end{equation}
for all $k\in \bn_0$.

We assume first that $0 < r <1$.
We need an estimate for the norm of
  $\Delta_t a$.
  \begin{equation}\label{eq:33}
  \begin{split}
        \norm{\Delta_t a}_\mA 
    \leq& \sum_{k=0}^M\norm{\Delta_t a_k}_\mA + \sum_{k=M+1}^\infty\norm{\Delta_t a_k}_\mA \\
    \leq& \sum_{k=0}^M\norm{\Delta_t a_k}_\mA + (M_\Psi+1)\sum_{k=M+1}^\infty\norm{ a_k}_\mA  \,,
  \end{split}
\end{equation}
where the value of $M$ will be chosen later.

  Lemma~\ref{lem:ch45} implies that
  \[
  \norm{\Delta_t a_k}_\mA\leq C 2^k \abs{t}\, \norm{a_k}_\mA 
  \]
  for all $k\in \bn $.
  Substituting back into (\ref{eq:33}) yields
  \begin{equation}
    \label{eq:bernsteindiff}
    \norm{\Delta_t a}_\mA 
    \leq C \Bigl(\sum_{k=0}^M2^k \abs t \norm{ a_k}_\mA  + \sum_{k=M+1}^\infty\norm{ a_k}_\mA  \Bigr) \,.
  \end{equation}
We use this relation for the estimation of the Besov seminorm.
\begin{align*}
  \abs a _{\besov p r \mA}
  &\asymp \biggl(\sum_{l=0}^\infty \bigl(2^{l r} \omega_{2^{-l}}(a) \bigr)^p \biggr)^{1/p}   \\
  & \ls \biggl(\sum_{l=0}^\infty 2^{l r p} 
        \Bigl( \sum_{k=0}^M2^k2^{-l}\norm{ a_k}_\mA  + \sum_{k=M+1}^\infty\norm{ a_k}_\mA \Bigr)^p \biggr)^{1/p}   \,.
\end{align*}
We split this expression into two parts and assume that $M=l$ in the inner sums.
\begin{align*}
  \abs a _{\besov p r \mA}
   \ls
   \biggl(\sum_{l=0}^\infty 2^{l (r-1) p} 
        \Bigl( \sum_{k=0}^l 2^k\norm{ a_k}_\mA \Bigr)^p \biggr)^{1/p}
      +
    \biggl(\sum_{l=0}^\infty 2^{l r p} 
        \Bigl(  \sum_{k=l+1}^\infty\norm{ a_k}_\mA \Bigr)^p \biggr)^{1/p}
        \,.
\end{align*}
We apply Hardy's inequalities  to both terms on the right hand side and obtain
\begin{align*}
  \abs a _{\besov p r \mA}
  &\ls 
\Bigl(\sum_{l=0}^\infty 2^{l (r-1) p} 
         2^{l p}\norm{ a_l}_\mA^p \Bigr)^{1/p}
      +
    \Bigl(\sum_{l=0}^\infty 2^{l r p} 
        \norm{ a_l}_\mA^p \Bigr)^{1/p}        \\
    &= 2 \Bigl(\sum_{l=0}^\infty 2^{l r p} 
        \norm{ a_l}_\mA^p \Bigr)^{1/p}  \, .
\end{align*}
As the representations $a=\sum_{k=0}^\infty a_k$ were arbitrary  we conclude that $\abs a _{\besov p r \mA} \ls \norm{a}_{\app p r \mA}$, using again the representation theorem.
  Next we consider the case $r = m+\eta $ for $m\in \bn_0 $ and $0<\eta
  <1$. The Bernstein inequality implies that
\[
\norm{\delta ^\alpha (a_k)}_{\mA } \leq C (2\pi 2^{k})^{|\alpha |}
\norm{a_k}_{\mA }
\]
for all $k\in \bn $ and $\alpha \in \bn ^d_0$. Consequently
$\sum _{k=0} ^\infty \delta ^\alpha a_k$ converges in $\mA $ for all
$\alpha $ with $|\alpha | \leq m$ and its sum must be $\delta ^\alpha(a)$ , as each $\delta _j$ is closed on $\mD (\delta ^\alpha )$. 
We now apply the above estimates %~\eqref{eq:33} and \eqref{eq:bernsteindiff} with
$\delta ^\alpha (a)$ instead of $a$  and deduce that 
$\delta ^\alpha (a) $ must be in $\besov p \eta \mA $ for $|\alpha |
\leq k$. Thus $a \in \besov p r \mA $. 

  If $r$ is an integer, then we have to use second order differences and a corresponding version of the mean
  value theorem. The argument is almost the same as above
  (see~\cite{Trigub04} for details in the scalar case). % for a basic version of 
\end{proof}

Combining Propositions~\ref{p:ch47} and~\ref{p:ch46}, we have completed the proof of Theorem~\ref{prop:jacksonbernstein}.

\section{Littlewood-Paley decomposition}
\label{sec:lp-decomposition-1}

\begin{proof}[Proof of Proposition \ref{prop:littl-paley-decomp-3}]
We include the derivation of the relevant results to keep the presentation self-contained. We follow~\cite{bergh76},  but we  use  approximation arguments where feasible.

We use some  obvious facts of the \DPU\ $(\varphi_k)_{k \geq -1}$.
By definition, $\supp \hat \varphi _ k = 2^k \supp \hat\varphi \subseteq \set{\omega \colon 2^{k-1} \leq \abs \omega _\infty \leq 2^{k+1}}$ for $k\geq 0$, and $\supp \varphi_{-1} \subseteq \set{\omega \colon \abs \omega_\infty \leq 1}$.
As the intersection of $\supp(\hat\varphi_k)$ with $\supp(\hat\varphi_{l})$ is nonempty only for $l\in \set{k-1,k,k+1}$ we obtain that
% \begin{equation}\label{eq:scalesupp}
%   \begin{split}
$    \varphi_k= \varphi_k * (\varphi_{k-1} + \varphi_k + \varphi_{k+1})$ if $k\geq 0$, and
    $\varphi_{-1}=\varphi_{-1}*(\varphi_{-1}+ \varphi_0)$.

Assume first that (\ref{eq:39}) holds. 
Then $\norm{\varphi_k *a}_\mA \leq C 2^{-rk}$, and so %the series
$\sum_{k=-1}^\infty\varphi_k*a$ 
is norm convergent in \mA. 
A standard weak type argument shows that the limit is actually $a$.
% We use a distributional argument to show that the limit is actually $x$.
% Actually, $y=\sum_{k=-1}^\infty\varphi_k*x$ implies that $G_{x',y}=\sum_{k=-1}^\infty \varphi_k*G_{x',x}$  in  $L^\infty(\brd)$ for all $x' \in \mA'$.
% Choose $\upsilon \in \schwartz$, then 
% \[
% \inprod{G_{x',y}, \upsilon} =\inprod{\sum_{k=-1}^\infty \varphi_k*G_{x',x}, \upsilon} =
% \inprod{\sum_{k=-1}^\infty \ \hat \varphi_k \hat G_{x',x}, \hat \upsilon} \,.
% \]
% where $\inprod{\phantom{x},\phantom{x}}$ denotes the dual pairing between $\schwartz '$ and $\schwartz$. 
% As $\sum_{k=-1}^\infty \ \hat \varphi_k =1$ in $\schwartz'$, %we conclude
% \[
% \inprod{\sum_{k=-1}^\infty \ \hat \varphi_k \hat G_{x',x}, \hat \upsilon}
% =\inprod{\hat G_{x',x}, \hat \upsilon} = \inprod{ G_{x',x}, \upsilon} \,.
% \]
% As the identity $\inprod{\widehat G_{x',y},\widehat  \upsilon} = \inprod{ G_{x',x}, \upsilon}$
%  is valid for all $\upsilon \in \schwartz$, we conclude that  $G_{x',x}= G_{x',y}$ for all $x' \in \mA'$, and this implies
% $y=x$.

For  $a \in \besov p r \mA$ and $m > \floor r$ we use %the norm equivalence  
$
  \norm{a}_{\besov p r \mA} \asymp \norm{a}_\mA + \bigl(\sum_{k=0}^\infty ( 2^{r k} \omega_{2^{-k}}^m(a) )^p \bigr)^{1/p}
$. 
 %where . %  for $x \in \besov p r \mA$. 
As 
$\norm{\Delta_t^m (\varphi_k * a)}_\mA \leq C_m \norm{\varphi_k * a}_\mA$ by Lemma~\ref{lem-modsmooth} (\ref{smoothordera}), and
$\norm{\Delta_t^m (\varphi_k * a)}_\mA \leq C' \abs{t}^m 2^{mk}\norm{\varphi_k * a}_\mA$ by repeated application of
 Lemma~\ref{lem:ch45} we conclude that 
 \begin{equation} \label{eq:itbernstein}
%   &\norm{\Delta_t^m (\theta * x)}_\mA  \leq C \min(1, \abs{t}^m) \norm{\theta * x}_\mA \,,  \\
   \norm{\Delta_t^m (\varphi_k * a)}_\mA \leq C_1 \min(1, \abs{t}^m 2^{mk}) \norm{\varphi_k * a}_\mA \, . 
 \end{equation}
As an immediate consequence we  obtain
\begin{equation}
  \label{eq:43}
 \omega_{\abs{t}}^m(a) \leq C_1 \sum_{k=-1}^\infty \min(1,t^m2^{mk}) \norm{\varphi_k * a}_\mA  \, ,
\end{equation}
and so
\begin{equation}
  \label{eq:61}
  2^{r j} \omega_{2^{-j}}^m(a) \leq C_1\sum_{k=1}^\infty 2^{(j-k) r}2^{k r}\min(1,{2^{-(j-k) m}}) \norm{\varphi_k * a}_\mA \,.
\end{equation}
The right hand side of this relation can be written as a convolution.
If we set
$
  u(l)= \min(1,2^{-lm})2^{lr} 
$
for $l \in \bz$, and 
$
  v(l)=  2^{lr} \norm{\varphi_l *a}_\mA$ if $l >-1$ and $0$ else,
then $u$ and $v$ are sequences in $\lone(\bz)$, and the right hand side of (\ref{eq:61}) is just $(u * v) (j)$.

So
$
\norm{\bigl(2^{r j} \omega_{2^{-j}}^m(a) \bigr)_{j \in \bn}}_{\ell^p(\bn)} \leq C \norm{u}_{\lone(\bz)} \norm{v}_{\lpz}
$, 
and this means that
\begin{equation}\label{eq: bsvsmlp}
  \norm{a}_{\besov p r \mA} \leq C  \Bigl( \sum_{k = -1}^\infty 2^{r k p}\norm{\varphi_k * a}_\mA^p \Bigr)^{1/p} \, ,  
\end{equation}
so (\ref{eq:39}) implies that $a \in \besov p r \mA$.

For the other inequality we  use %the norm
$
\norm{a}_{\besov p r \mA} \asymp \norm{a}_\mA + \sum_{\abs \alpha = m} \norm{\delta^\alpha(a)}_{\besov p {r-m} \mA}
$
with $m < r \leq m+1$.% and the ``discrete'' form of $\norm{\phantom x}_{\besov p {r-m} \mA}$.

First we show that %e following inequalities.
\begin{equation}
  \label{eq:62}
  \norm{\varphi_k * a}_\mA \leq C 2^{-mk} \norm{\varphi_k *\delta^\alpha(a)}_\mA \, , \quad m= \abs \alpha
\end{equation}
and
\begin{equation}
  \label{eq:63}
  \norm{\varphi_k *\delta^\alpha(a)}_\mA \leq C \omega_{2^{-k}}^{2} (\delta^\alpha a).
\end{equation}
For the proof of these relations choose an even function $\Phi \in S(\brd)$ such that
  $\hat \Phi \equiv 1$ on $\supp \hat \varphi_0$, and  
  $\hat \Phi \equiv 0$ in a neighbourhood of 0.
% and $\hat \Phi(\omega)= \hat \Phi(-\omega)$ for all $\omega \in \brd$, so $\hat\Phi$ and $\Phi$ are even functions.
Set $\Phi_k(t)=2^{kd}\Phi(2^k t)$, then $\norm{\Phi_k}_1=\norm{\Phi}_1$ and $\Phi_k * \varphi_k = \varphi_k$.
The  function
$
\hat \eta ^{(\alpha)}\colon \omega \to (2 \pi i \omega)^{-\alpha}{\hat \Phi (\omega)} 
$
is an element of $\schwartz$. Again, if we set $\eta_k ^{(\alpha)} (t) =2^{kd}\eta ^{(\alpha)}(2^k t)$, then $\norm{\eta_k ^{(\alpha)}}_1=\norm{\eta ^{(\alpha)}}_1$.
Then 
$
\hat \Phi_k(\omega)=\hat \Phi (2^{-k}\omega)= 2^{-k \abs \alpha} (2 \pi i \omega)^\alpha \hat \eta_k^{(\alpha)}(w) \,,
$
and so, assuming that $\abs \alpha =m$, we obtain
$
\hat \varphi_k(\omega) = 2^{-k m} \hat \eta_k ^{(\alpha)}(\omega) (2 \pi i \omega)^\alpha \hat\varphi_k(\omega) 
$
for all $\omega\in\brd$, which implies
\begin{equation*}
  \varphi_k * a = 2^{- k \, m} \eta_k ^{(\alpha)} * \delta^\alpha(\varphi_k * a)= 2^{- k \, m} \eta_k ^{(\alpha)} * \varphi_k *\delta^\alpha( a),
\end{equation*}
the last equality by \eqref{eq:derivderiv}. Now (\ref{eq:62}) follows immediately.

For the proof of (\ref{eq:63}) set $y = \delta^\alpha(a)$ and $y_k=\varphi_k * y = \Phi_k * \varphi_k * y = \Phi_k * y_k$.
We obtain
\begin{align*}
\varphi_k * y &= \Phi_k*y_k 
= \int_{\brd} \Phi_k(t) \psi_{-t}(y_k) \, dt \\
&= \tfrac{1}{2} \int_{\brd} \Phi_k(t) \bigl\{\psi_{-t}(y_k) -2 y_k +\psi_t(y_k)  \bigr\} \, dt
= \tfrac{1}{2} \int_{\brd} \Phi_k(t) \psi_{-t}\Delta^2_t(y_k) \, dt \,,
\end{align*}
as $\int_{\brd} \Phi_k =0$ and $\Phi_k(-t)=\Phi_k(t)$. Changing variables we obtain
\begin{align*}
  \varphi_k * y =\tfrac{1}{2} \int_{\brd} \Phi(u) \psi_{-2^{-k}u}\Delta^2_{2^{-k}u}(y_k) \, dt
=\tfrac{1}{2} \int_{\brd} \Phi(u) \psi_{-2^{-k}u}\bigl( \varphi_k*\Delta^2_{2^{-k}u}(y) \bigr) \, dt.
\end{align*}
Taking norms we get
\begin{align*}
\norm{ \varphi_k * y}_{\mA} &\leq \tfrac{M_\psi}{2}  \int_{\brd} \abs{\Phi(u)} \norm{ \varphi_k}_1 \omega^{2}_{2^{-k}\abs u}(y)  \, dt.\\
&\leq \tfrac{M_\psi}{2} \norm{ \varphi_0}_1 \int_{\brd} \abs{\Phi(u)}  (1+ \abs u ^2) \omega^{2}_{2^{-k}}(y)  \, dt\\
&\leq C \omega^{2}_{2^{-k}}(y) \,,
\end{align*}
where the estimate for  $\omega^2_{2^{-k}\abs u}(y)$ follows from Lemma~\ref{lem-modsmooth}. This is what we wanted to show. 

The proof of the reverse inclusion now follows by putting (\ref{eq:62}) and (\ref{eq:63}) together.
\begin{equation*}
  2^{r k } \norm{\varphi_k *a}_\mA \leq C 2^{(r- m) k} \norm{\varphi_k * \delta^\alpha(a)}_\mA \leq C 2^{(r -m) k} \omega^{2}_{2^{-k}}(\delta^\alpha(a)) \,,
\end{equation*}
and so
\begin{equation}\label{eq:lplsbsv}
  \begin{split}
    \sum_{k=-1}^\infty 2^{r p k } \norm{\varphi_k *a}^p_\mA 
    &\leq C \bigl(\norm{a}_\mA^p+\sum_{k=0}^\infty 2^{(r -m)p k} \omega^{2}_{2^{-k}}(\delta^\alpha(a))^p\bigr) \\
    &\leq C' (\norm{a}_\mA^p + \abs{\delta^\alpha(a)}^p_{\besov p {r-m} \mA})
    \leq C" \norm{a}^p_{\besov p {r-m} \mA} \,.
  \end{split}
\end{equation}
We have shown that $a \in \besov p {r} \mA$ implies (\ref{eq:39}).
The norm equivalence follows from (\ref{eq: bsvsmlp}) and (\ref{eq:lplsbsv}).
\end{proof}

\bibliography{almost_diagonal} 
\bibliographystyle{abbrv}
\end{document}